\documentclass[reqno, a4paper,11pt]{amsart}
\usepackage[english]{babel}
\usepackage[utf8]{inputenc}

\usepackage[dvipsnames]{xcolor}
\usepackage{enumerate}
\usepackage{amsthm,amssymb,mathrsfs}
\usepackage{amsmath}
\usepackage{graphicx,epsfig}
\usepackage{hyperref}
\usepackage{amsfonts, mathtools, subfig, makebox, dsfont, bm}
\usepackage{comment}
\usepackage{amsbsy}

\hypersetup{
    colorlinks=true, 
    linkcolor=NavyBlue,  
    citecolor=NavyBlue,  
    urlcolor=NavyBlue    
}

\newtheorem{theo}{Theorem}[section]
\newtheorem{lemma}[theo]{Lemma}
\newtheorem{prop}[theo]{Proposition}
\newtheorem{cor}[theo]{Corollary}

\newtheorem{hyp}[theo]{Hypothesis}
\newtheorem{hyps}[theo]{Hypotheses}
\newtheorem{defi}[theo]{Definition}
\theoremstyle{definition}

\newtheorem{rem}[theo]{Remark}

\def\C{{\mathbb{C}}}
\def\a{{\bm{a}}}
\def\R{{\mathbb{R}}}

\def\N{\mathbb N}

\def\K{\mathbb K}
\def\ph{\varphi}
\def\f{\bm{f}}
\def\g{\bm{g}}
\def\uu{\bm{u}}

\def\eps{\varepsilon}

\newcommand{\D}{\nabla}

\newcommand{\norm}[1]{\left\Vert#1\right\Vert}

\newcommand{\sign}[1]{\operatorname{sign}\left(#1\right)}

\DeclareMathOperator*{\esssup}{ess\,sup}

\DeclareMathOperator{\re}{\mathrm{Re}}

\allowdisplaybreaks

\begin{document}
\numberwithin{equation}{section}

\title[Maximal inequalities and Riesz transforms]{Maximal inequalities and Riesz transforms  for vector-valued magnetic Schr\"odinger operators}
\author[D. Addona]{Davide Addona}
\author[V. Leone]{Vincenzo Leone}
\author[L. Lorenzi]{Luca Lorenzi}
\author[E.M. Ouhabaz]{El Maati Ouhabaz}

\author[A. Rhandi]{Abdelaziz Rhandi}

\address{D. Addona, L. Lorenzi: Plesso di Matematica, Dipartimento di Scienze Matematiche, Fisiche e Informatiche, Università di Parma, Parco Area delle Scienze 53/A, 43124 Parma, Italy}
\email{davide.addona@unipr.it, luca.lorenzi@unipr.it}

\address{V. Leone, A. Rhandi: Dipartimento di Matematica, Università degli Studi di Salerno, Via Giovanni Paolo II, 132, 84084 Fisciano (SA), Italy}
\email{vleone@unisa.it, arhandi@unisa.it}

\address{E.M. Ouhabaz: Institut de Math\'ematiques de Bordeaux, Universit\'e de Bordeaux, UMR CNRS 5251. 
351, cours de la Liberation, 33405 Talence, France}
\email{Elmaati.Ouhabaz@math.u-bordeaux.fr}

\thanks{This work  is  supported by the project ``Elliptic and parabolic problems, heat kernel estimates and spectral theory" CUP D53D23005580006, funded by European Union
Next Generation EU within the PRIN 2022 program (D.D. 104 - 02/02/2022 Ministero dell’Università e della Ricerca) and by the project ``Operatori associati a Sistemi di equazioni di Kolmogorov in spazi $L^p$'' CUP E53C25002010001, founded by G.N.A.M.P.A. of the Italian Istituto Nazionale di Alta Matematica (I.N.d.A.M.). 
The authors D.A., V.L., L.L., A.R. are also members of G.N.A.M.P.A. of the Italian Istituto Nazionale di Alta Matematica (I.N.d.A.M.). }

\keywords{Magnetic Schr\"odinger operators, Riesz transform, maximal inequalities, Reverse H\"older class.}
 
\subjclass[2020]{35J47, 
47D08, 
42B20, 
42B37} 
 
\begin{abstract}
We consider vector-valued magnetic Schr\"odinger operators $-\bm \Delta_{\a}+V$ with magnetic potential $\bm a \in L^2_{\mathrm{loc}}(\R^d;\R^d)$ and electric potential $V$ given by a matrix-valued function whose entries belong to $L^1_{\mathrm{loc}}(\R^d)$. We prove maximal inequalities in $L^p(\R^d;\C^m)$, $p\in[1,\infty)$ and the boundedness of the Riesz transforms $(\nabla - i\bm a)(-\bm \Delta_{\bm a}+V)^{-\frac12}$ and $V^{\alpha}(-\bm \Delta_{\bm a}+V)^{-\alpha}$ on $L^p(\R^d;\C^m)$ for every $p \in (1,2]$ and every $\alpha\in[0,1/p]$.
\end{abstract}
	
\maketitle

\section{Introduction}

Vector-valued Schr\"odinger operators of the form $-\bm \Delta+V$ in $L^p(\R^d;\R^m)$, $p\in[1,\infty)$, for a suitable 
class of symmetric matrix-valued electric potentials $V$ are studied in \cite{ALLR24}. The aim there is to prove maximal inequalities for the operator $-\bm \Delta+V$ and generation results for the semigroups. More precisely, it is shown that the realization $T_p$ of $\bm\Delta-V$ in $L^p(\R^d;\R^m)$, $p\in[1,\infty)$, with maximal domain, generates a strongly continuous semigroup, which is also analytic for $p\in(1,\infty)$. Under reverse H\"older estimates on the minimal eigenvalue of $V$, it is also proved that the domain of $T_p$ coincides with the minimal domain in $L^p(\R^d;\R^m)$ for $p\in[1,\infty)$. 

The purpose of the present article is twofold. Firstly, we consider more general second order differential operators, including vector-valued Schr\"odinger operators with magnetic fields. Secondly, we improve the results from \cite{ALLR24} by  removing some assumptions on the potential. Our approach, which is quite different from that of the aforementioned article, yields more precise estimates. In addition, we study the boundedness on $L^p(\R^d; \C^m)$ for the Riesz transforms  $(\nabla - i\bm a)(-\bm \Delta_{\bm a}+V)^{-\frac12}$ and $V^{\alpha}(-\bm \Delta_{\bm a}+V)^{-\alpha}$.   

To give an overview of some results we first introduce some notation.  
Let $\a$ be a vector field in $\R^d$ and $V$  a nonnegative symmetric matrix with locally integrable entries. We define the  magnetic Schr\"odinger operator $\pmb{\mathscr{A}}$, whose action on vector-valued smooth functions 
$\bm{u} =(u_1, \dots,u_m) \colon\R^d\to\C^m$ is given, in the sense of distributions, by
\begin{equation}
\label{eq:operatore:distr}
(\pmb{\mathscr{A}}\bm{u})_h = 
-(\nabla-i\a)^*(\nabla-i\a)u_h+(V\bm u)_h=- \sum_{k=1}^d(\partial_k-ia_k)^2u_h+ \sum_{k=1}^m v_{hk}u_k
\end{equation}
for every $h\in\{1,\dots,m\}$, where $i$ is the imaginary unit. 
Since several operators are involved in this paper we adopt the following notation. 
\begin{description}
\item[$\bm{A}$] is the self-adjoint realization of $\pmb{\mathscr{A}}$. It is the operator associated with a symmetric 
sesquilinear  form on $L^2(\R^d;\C^m)$; 
\item[$\bm{L}$] denotes the operator $\bm{A}$ in the case where $V\equiv 0$, i.e., when $\pmb{\mathscr{A}}=-\bm\Delta_{\bm a}$. If $m=1$, then we simply write $L$.
\item[$H$] is the operator $-\Delta + \lambda_V$ with $\lambda_V(x)$ being the smallest eigenvalue of the matrix $V(x)$. It is 
classically defined  on $L^2(\R^d;\C)$ by a sesquilinear symmetric form.
 \end{description}
All the above operators are nonnegative and self-adjoint. We denote their associated  semigroups by 
$(e^{-t \bm{A}})_{t\ge0}$, $(e^{-t \bm{L}})_{t\ge0}$ and 
$(e^{-t H})_{t\ge0}$, respectively, on the corresponding $L^2$-space. We show that  each of these semigroups is  $L^\infty$-contractive. 
Therefore, they induce strongly continuous semigroups on $L^p$. If $\bm S\in\{\bm A, \bm L\}$, then we denote by $\bm S_p$ the infinitesimal generator of the extension of the semigroup $(e^{-t{\bm S}})_{t\geq0}$ to $L^p(\R^d;\C^m)$ for $p\in[1,\infty)$. A similar notation is also used  for $S\in\{H,L\}$.

Throughout the paper, unless otherwise specified, we make the following standing assumptions.

\begin{hyps}\label{hyp-operatore}

\begin{enumerate}[\rm (i)] 
\item 
The magnetic potential $\a$ belongs to $L^2_{\rm loc}(\R^d;\R^d)$.
\item 
The electric potential $ V\colon \R^d\to\R^{m\times m}$ is a matrix-valued function such that, for almost every $x\in\R^d$, the entries $v_{ij}$ satisfy $ v_{ij}(x) = v_{ji}(x)$, $v_{ij}\in L^1_{\rm loc}(\R^d)$ for all $ i,j\ \in\{1,\dots,m\}$. We also assume that the smallest  eigenvalue $\lambda_V(x)$ of  $V(x)$ is nonnegative for almost every $x\in\R^d$.
\end{enumerate}
\end{hyps}

Under these assumptions and adopting the above notation we  summarize below some of our main results.  
\begin{theo}
We have the following assertions. 
\begin{enumerate}[\rm (a)]
\item {\bf The diamagnetic inequality.} The semigroup $(e^{-t \bm{A}})_{t\ge0}$ is dominated by the semigroup $(e^{-tH})_{t\ge0}$. That is, for every $t\geq0 $ and every $\f\in L^2(\R^d;\C^m)$,
\begin{equation*}
\| (e^{-t \bm{A}} {\bm f}) (x)\| \le (e^{-t H} \| \bm f \|)(x), \qquad \textrm{a.e. }x\in\R^d. 
\end{equation*}
\item{\bf The $L^1$-maximal inequality.} Suppose in addition that the largest eigenvalue $\Lambda_V(x)$ of $V(x)$ is comparable to $\lambda_V(x)$ in the sense that there exists a positive constant $\kappa$ such that 
\begin{equation}\label{big-small}
\lambda_V(x) \le \Lambda_V (x) \le \kappa \lambda_V(x) \qquad\;\, a.e.\, x \in \R^d.
\end{equation}
Then, 
\begin{equation*}
\norm{\bm\Delta_{\a} \uu}_1\leq (\kappa+1)\|\bm{A}_1\bm u\|_1, \qquad\;\,  \norm{V\uu}_1 \le \kappa \norm{\bm{A}_1 \uu}_1,\quad\;\,\uu\in D(\bm{A}_1).
\end{equation*}
\item{\bf The $L^p$-maximal inequality.} Suppose again that \eqref{big-small} holds and that $\lambda_V$ belongs to the reverse H\"older class $B_q$ for some $ q\in (1,\infty)$. Then there exists a positive constant $c_p'$ such that, for every $p\in [1,q]$,
\begin{equation*}
\norm{\bm\Delta_{\a} \uu}_p\leq (\kappa c_p'+1)\|\bm{A}_p\bm u\|_p, \qquad  \norm{V\uu}_p \le \kappa c_p'\norm{\bm{A}_p \uu}_p,\quad\;\,\uu\in D(\bm{A}_p).
\end{equation*}
\item{\bf First-order Riesz transform.} Suppose  \eqref{big-small}  and  that $\a\in L^q_{\rm loc}(\R^d;\R^d)$ for some $q\in (2,\infty)$ and $\operatorname{div}\a\in L^s_{\rm loc}(\R^d)$ for some $s\in (1,\infty)$. Then, for every $p\in(1,2]$ the operators ${\bm L}^{\frac12}{\bm A}^{-\frac12}$ and $\nabla_\a{\bm{A}}^{-\frac12}$ admit bounded extensions to  $L^p(\R^d;\C^m)$.
\item{\bf  Zero-order Riesz transform.} Suppose again \eqref{big-small}.  For every $p\in (1,2]$ and every $\alpha\in[0,1/p]$, the operator $V^{\alpha}{\bm A}^{-\alpha}$ admits a bounded extension to $L^p(\R^d;\C^m)$.
\end{enumerate}
\end{theo} 
 
The proof of the diamagnetic inequality is based on criteria for the domination of semigroups proved in \cite{Ouhabaz96} 
(and in \cite{Ouhabaz99} for vector-valued operators). See also \cite[Theorem 2.30]{Ouhabaz}. The domination 
in assertion $(a)$ implies in particular that the semigroup $(e^{-t { \bm A}})_{t\ge0}$ extends to a contraction semigroup 
on $L^p(\R^d; \C^m)$ for all $p \in [1, \infty)$. Thus, we can define ${ \bm A}_p$ as minus the generator of the semigroup
on $L^p(\R^d; \C^m)$. The $L^1$-maximal inequality is heavily based on the diamagnetic inequality and differs from the 
proof  in \cite{ALLR24}. We note that in the latter reference, the authors consider the case ${\bm a} \equiv \bm 0$ and 
assume in addition that $v_{ij} \le 0$ for all $i,j\in\{1,\dots,m\}$ such that $i \not=j$. Concerning the Riesz transform 
$\nabla_\a{\bm{A}}^{-\frac12}$ on $L^p(\R^d; \C^m)$, 
our proof is partially inspired by  \cite{dziubanski} who treats the case of scalar-valued Schr\"odinger operators $-\Delta + V$. 
The aim in \cite{dziubanski}  is to prove dimension-free and potential-free estimates on $L^p(\R^d; \C)$ for 
$\nabla(-\Delta + V)^{-\frac12}$ with nonnegative potentials $V$. This was achieved there, in dimensions $d\ge 3$, by elegant and quite elementary 
arguments which avoid appealing to  the theory of singular integral operators. We emphasize that, in the scalar case $m=1$, we obtain dimension-free and potential-free estimates for the Riesz transforms of magnetic Schr\"odinger operators, with no restriction on the dimension. To the best of our knowledge, this result is new.
We believe that one can adapt the 
proof from \cite{DMc}, based on singular integrals,  to our general vector-valued context. 
However, the approach we adopt here gives some additional information and in particular it provides that the estimates of the Riesz transform depend on $V$ and possibly on the dimension only through the constant $\kappa$ in  \eqref{big-small}. We  mention that the question of the Riesz transform for Schr\"odinger operators has been studied before in the scalar-valued case. In fact, the cases $p\in(1,2]$ and $p=1$ (in a weak-type  sense) were treated in \cite{sikora}, using the wave equation method and in \cite{DOY06} using Hardy space techniques. The papers \cite{BenAli1, BenAli2} treat electric potential satisfying reverse H\"older estimates, 
following the approach of \cite{auscher-benali:2007}  and \cite{Shen} in the absence of a magnetic field. The zero-order Riesz type transforms $V^\alpha (-\Delta + V)^{-\alpha}$ were considered in \cite{KW25} in the scalar-valued case and nonnegative potentials $V$. We follow and adapt some of their arguments to our general setting. The boundedness on $L^p(\R^d; \C)$ of the zero-order Riesz transform  is proved in  \cite{KW25} by Stein's complex  interpolation theorem. We provide in the appendix a vector-valued counterpart of this well-known result. It is worth emphasizing that, when working with complex interpolation techniques, the boundedness of the imaginary powers of the operator $\bm A$ in $L^p$-spaces plays a crucial role. 
 This property follows from a recent result in \cite{ALLR-2}, combined with the fact that the semigroup generated by $\bm A$ satisfies Gaussian-type estimates.

\subsection*{Organization of the paper}
In Section \ref{sec:msy}, we define  the operator $\bm A$ via the theory of sesquilinear forms. 
We also establish a key vector-valued result on the magnetic gradient, which allows us to prove, in 
Section \ref{sec:diam:sg}, the diamagnetic inequality.  
Maximal inequalities in $L^p(\R^d;\C^m)$ are established in Section \ref{sec:max_ineq_Lp}, 
assuming some reverse H\"older estimates on $\lambda_V$ when $p>1$. The case $p=1$ plays a fundamental 
role in  Section \ref{sec:RT}, where we prove the boundedness of the Riesz transform $(\nabla-i\a)\bm{A}^{-\frac12}$ in 
$L^p(\R^d;\C^m)$ for  $p\in(1,2]$ under additional assumptions on $\a$. 
We stress that such additional assumptions enable us to prove the Kato's distributional inequality for the magnetic Laplacian, 
obtained in \cite{kato:1972} under the stronger condition $\a\in C^1(\R^d;\R^d)$.    
Finally, the boundedness of the operator $V^{\alpha}{\bm A}^{-\alpha}$ on $L^p(\R^d;\C^m)$, 
$p\in(1,2]$ and $\alpha\in[0,1/p]$, is discussed in Section \ref{sec:0:RT}. 
The latter result relies on a vector-valued Stein's complex interpolation theorem which we present in the Appendix \ref{appendix}.

\subsection*{Notation} Let $\K$ be either $\R$ or $\C$. For every $\xi\in\K^m$, 
$\|\xi\| = \langle\xi,\xi\rangle^{\frac12}$ stands for the Euclidean norm of $\K^m$. If $\bm f:\R^d\to \K^m$, then $\|\bm f\|$ denotes the scalar function defined as $\|\bm f\|(x)=\|\bm f(x)\|$ for every $x\in\R^d$. Moreover, 
${\rm sign}(\f) = \f\|\f\|^{-1}{\mathds 1}_{\{\f\ne \bm{0}\}}$, where, for every $\Omega\subset\R^d$, $\mathds 1_{\Omega}$ is its characteristic function. If $\f:\R^d\to\K^m$ is a vector-valued function, we denote its components $f_1,\ldots,f_m$. Similarly, if $(\f_{\varepsilon})_{\varepsilon\in E}$ is a set of functions which take values in $\K^m$, then $f_{\varepsilon,1},\ldots,f_{\varepsilon,m}$ denote the components of the function $\f_{\varepsilon}$.  

$C_c^{\infty}(\R^d;\mathbb{K}^m)$ denotes the space of compactly supported and infinitely differentiable functions $\f:\R^d\to\mathbb{K}^m$ ($d,m\in\N$) and $L^p(\R^d;\mathbb{K}^m)$ denotes the space of (the equivalence classes of) measurable $\mathbb K^m$-valued functions, for which
$\|\f\|_p^p= \int_{\R^d}\|\f(x)\|^pdx$ is finite, if $1 \leq p <\infty$, and such that
$\|\f\|_{\infty}= {\rm ess\sup}_{x\in\R^d}\|\f(x)\|<\infty$, if $p=\infty$. $L^p_{\rm loc}(\R^d;\mathbb{K}^m)$ is the set of (equivalence classes of) measurable functions that belong to $L^p(\Omega;\mathbb{K}^m)$ for every bounded measurable subset $\Omega$ of $\R^d$.
For $k\in\N$ and $p$ as above, $W^{k,p}(\R^d;\mathbb{K}^m)$ denotes the Sobolev space of order $k$, i.e., the space of functions $\f\in L^p(\R^d;\mathbb{K}^m)$ such that the distributional derivative $\frac{\partial^{\beta}\f}{\partial x^{\beta}}$ belongs to $L^p(\R^d;\mathbb{K}^m)$ for any multi-index $\beta$ with length at most  $k$. When $p=2$, we write $\textup{H}^k(\R^d;\mathbb{K}^m)$ instead of $W^{k,2}(\R^d;\mathbb{K}^m)$. If $\mathbb K=\mathbb R$ and $m=1$, then in the function spaces we omit the codomain. Moreover, we simply write $\partial_k$ instead of $\frac{\partial}{\partial x_k}$.

For $ q\in(1,\infty)\cup\{\infty\}$, the reverse H\"older class $B_q$ is the set of almost everywhere positive functions $w\in L_{\rm loc}^q(\R^d)$ for which
there exists a positive constant $c$ such that 
\begin{equation*}
\left(\frac1{|Q|}\int_Q w^q(x)dx\right)^{\frac1q}\le\frac c{|Q|}\int_Qw(x)dx
\end{equation*}
for all cubes $Q\subset\R^d$. For $q=\infty$, the left-hand side in the last condition is replaced by $\esssup_{x\in Q}w(x)$.
The minimum of the positive constants $c$ such that the previous inequality is satisfied is denoted by $[w]_{B_q}$. H\"older's inequality shows that 
$[w]_{B_q}\ge 1$ for every $w\in B_q$ and $B_q\subseteq B_p$ with $[w]_{B_p}\leq [w]_{B_q}$ for every $p\in(1,q)$.

The space of bounded linear operators from a Banach space $X$ into another Banach space $Y$ is denoted by $\mathcal B(X,Y)$. If $Y=X$ we simply write $\mathcal B(X)$. 

Let $\Omega$ be a subset of $\R^d$. We denote by $\mathscr{L}(\Omega)$ the set of all Lebesgue measurable subsets of $\Omega$ with finite Lebesgue measure.

The open ball in $\R^d$ centered at $x$ with radius $r$ is denoted by $B(x,r)$ and $|B(x,r)|$ stands for its Lebesgue measure. When $x=0$, we simply write $B(r)$ instead of $B(0,r)$.

\section{Magnetic Schr\"odinger systems}
\label{sec:msy}

Given a vector field $\a\colon\R^d\to\R^d$, we introduce the magnetic gradient $\nabla_\a = \nabla -i\a$ and the magnetic Laplace operator $\Delta_{\a} = -\nabla_{\a}^*\nabla_{\a}$ whose action on a scalar smooth function $f\colon\R^d\to\C$ is $\Delta_{\a} f = \sum_{k=1}^d(\partial_k-ia_k)^2f$. Denoting by $\bm\Delta_\a\bm u $ the vector $(\Delta_\a u_1,\dots,\Delta_\a u_m)$ for a 
vector-valued smooth function $\bm{u}\colon\R^d\to\C^m$, it follows that the operator $\pmb{\mathscr{A}}$, introduced in \eqref{eq:operatore:distr}, reads as
\begin{equation*}
\pmb{\mathscr{A}}\bm{u}\coloneqq-{\bm\Delta}_{\a}\bm{u}+V\bm{u}.  
\end{equation*}

We begin with an auxiliary result on the magnetic gradient $\nabla_{\a}$. 

\begin{lemma}
\label{lem:first:diamagnetic}
Let $\a$ satisfy Hypothesis $\ref{hyp-operatore}(i)$ and let ${\uu}\in L^2(\R^d;\C^m)$ be such that $\nabla_{\a}u_i\in L^2(\R^d;\C^d)$ for every $i\in\{1,\dots,m\}$. Then, $\|\uu\|$ belongs to $\textup{H}^1(\R^d)$ and 
\begin{equation}\label{eq:diamagnetic:01}
\|(\D\|\uu\|)(x)\|^2 \le\sum_{j=1}^m\|(\D_{\a} u_j)(x)\|^2,\qquad\;\,
\textit{a.e.}~x\in\R^d.
\end{equation}
\end{lemma}

\begin{proof}
Let $\bm u$ be as in the statement. It suffices to prove formula \eqref{eq:diamagnetic:01}, since it immediately implies that the gradient of $\|\uu\|$ belongs to $L^2(\R^d;\R^d)$, so that $\|\uu\|\in \textup{H}^1(\R^d)$.

To prove formula \eqref{eq:diamagnetic:01}, we fix $j\in\{1,\ldots,m\}$ and begin by observing that, 
since $\nabla u_j = \nabla_{\a} u_j + i\a u_j$, $\uu\in L^2(\R^d;\C^m)$ and $\a\in L^2_{\rm loc}(\R^d;\R^d)$, H\"older's inequality shows that the distributional gradient $\nabla u_j$ belongs to $L^1_{\rm loc}(\R^d;\C^d)$. It thus follows that $\|\uu\|$ belongs to $W^{1,1}_{\rm loc}(\R^d)$. Moreover, 
\begin{equation}
\D\|\uu\| = \frac1{\|\uu\|}\sum_{j=1}^m \re(\overline{u_j}\,\D u_j){\mathds 1}_{\{\uu\ne\bf0\}}    
\label{eq:grad:norma}
\end{equation}
almost everywhere in $\R^d$. This formula can be proved arguing as in 
\cite[Lemma 2.4]{MR18}, which deals with the case $p\in (1,\infty)$. 
The main steps are as follows:
\begin{enumerate}[\rm (i)]
\item 
By an approximation argument, via smooth functions, one proves that, 
for every $\varepsilon>0$, the function $w_{\varepsilon}=(\|\uu\|^2+\varepsilon^2)^{\frac{1}{2}}-\varepsilon$ belongs to $W^{1,1}_{\rm loc}(\R^d)$, and that 
\begin{eqnarray*}
\nabla w_{\varepsilon}=\frac{1}{w_{\varepsilon}}\sum_{j=1}^m{\rm Re}(\overline{u_j}\nabla u_j)    
\end{eqnarray*}
almost everywhere in $\R^d$.
\item 
By the dominated convergence theorem, $w_{\varepsilon}$ converges to $\|\uu\|$ in $L^1(K)$, and $\nabla w_{\varepsilon}$ to the right-hand side of 
\eqref{eq:diamagnetic:01} in $L^1(K;\R^d)$ as $\varepsilon$ tends to $0^+$, for every compact set $K\subset\R^d$. 
\end{enumerate}

Since $\re(\overline{u_j}\nabla u_j) 
= \re(\overline{u_j}\nabla_{\a}u_j)$, formula \eqref{eq:diamagnetic:01} can be rewritten as
\begin{equation*}
\D\|\uu\|= \frac1{\norm{\uu}}\sum_{j=1}^m \re(\overline{u_j}\,\D_{\a} u_j){\mathds 1}_{\{\uu\ne\bf0\}}
\end{equation*}
almost everywhere in $\R^d$.

Finally, the Cauchy-Schwarz inequality yields 
\begin{align*}
\|\uu\| \|\nabla\|\uu\|\|\le\sum_{j=1}^m|\re(\overline{u_j}\,\D_{\a} u_j)|
\le\|\uu\|\bigg (\sum_{j=1}^m|\D_{\a} u_j|^2\bigg )^{\frac12}
\end{align*}
pointwise in $\R^d$ and formula \eqref{eq:diamagnetic:01} is proved. 
\end{proof}

We now introduce the sesquilinear form $\mathfrak{a}:D(\mathfrak{a})\times D(\mathfrak{a})\to\C$, defined by
\begin{equation*}
\mathfrak{a}(\bm{u},\bm{v})=\sum_{i=1}^m\int_{\R^d}\langle \nabla_{\a} u_i(x),\nabla_{\a} v_i(x)\rangle dx+\int_{\R^d}\langle V(x) \bm{u}(x),\bm{v}(x) \rangle dx, \qquad\;\, \bm{u},\bm{v}\in D(\mathfrak{a}),    
\end{equation*}
where
\begin{equation*}
  D(\mathfrak{a})\coloneqq\{\bm{f}\in L^2(\R^d;\C^m)\colon \nabla_{\a}f_i\in L^2(\R^d;\C^d) \text{ for } i\in\{1,\dots,m\}, V^{\frac{1}{2}}\bm{f}\in L^2(\R^d;\C^m)\} 
\end{equation*}
is endowed with the norm $\|\uu\|_{D(\mathfrak{a})}^2
=\|\uu\|_2^2+\mathfrak{a}(\uu,\uu)$.

In the next propositions we summarize the main properties of this form. We begin with a preliminary lemma, whose straightforward proof is left to the reader.

\begin{lemma}
Let $\a$ satisfy Hypothesis $\ref{hyp-operatore}(i)$. Then, the space 
\begin{equation*}
\textup{H}_{\a}^1(\R^d;\C^m)\coloneqq\{\bm{f}\in L^2(\R^d;\C^m)\colon \nabla_{\a}f_i\in L^2(\R^d;\C^d) \textit{ for } i\in\{1,\dots,m\}\}
\end{equation*}
is a complete Hilbert space if endowed with the norm $\|\bm f\|_{\textup{H}^1_{\a}}^2:=\|\bm f\|_2^2+\|\nabla_{\a}\bm f\|_2^2$ for every $\bm f\in \textup{H}^1_{\bm a}(\R^d;\C^m)$.
\end{lemma}

\begin{prop}\label{prop:prop:a}
The sesquilinear form $\mathfrak{a}$ is densely defined, accretive, symmetric, continuous and closed.
\end{prop}
\begin{proof}
It is immediate that $C_c^\infty(\R^d;\C^m)\subset D(\mathfrak{a})$. Hence, the form $\mathfrak{a}$ is densely defined. The symmetry and the accretiveness of the form are straightforward and follow from Hypothesis \ref{hyp-operatore}(ii), which ensures that the potential $V$ is real, symmetric and accretive.  Similarly, the continuity of the form $\mathfrak{a}$ is an easy consequence of the Cauchy-Schwarz inequality. 

To conclude the proof, we need to show that the form is closed. Let $(\f_n)_{n\in\N}$ be a Cauchy sequence in $(D(\mathfrak{a}),\|\cdot\|_{D(\mathfrak{a})})$. Then, 
$(\f_n)_{n\in\N}$ is a Cauchy sequence in $\textup{H}_{\a}^1(\R^d;\C^m)$ and $(V^{\frac{1}{2}}\f_n)_{n\in\N}$ is a Cauchy sequence in $L^2(\R^d;\C^m)$. Since both spaces are complete, 
there exist $\f\in \textup{H}_{\a}^1(\R^d;\C^m)$ and $\bm g\in L^2(\R^d;\C^m)$ such that $(\f_n)_{n\in\N}$ converges to $\f$ in $\textup{H}_{\a}^1(\R^d;\C^m)$ and $(V^{\frac12}\f_n)_{n\in\N}$ converges to $\bm g$ in $L^2(\R^d;\C^m)$, as $n$ tends to $\infty$. By a standard argument, we obtain that $\bm g=V^{\frac{1}{2}}\f$. Hence, $\f\in D(\mathfrak{a})$ and $\f_n$ converges to $\f$ in $D(\mathfrak{a})$ as $n$ tends to $\infty$.
\end{proof}

As an immediate consequence, the operator $\pmb{\mathscr{A}}$ admits a self-adjoint realization $\bm{A}$ on $L^2(\R^d;\C^m)$ which is associated to the sesquilinear form $\mathfrak{a}$ and the operator $-\bm{A}$ is the generator of a strongly continuous and analytic semigroup of contractions on $L^2(\R^d;\C^m)$ which we denote by $(e^{-t\bm{A}})_{t\ge0}$. Let us prove that the space $C_c^\infty(\R^d;\C^m)$ is a core for the form $\mathfrak{a}$.

\begin{prop}
The space $C_c^\infty(\R^d;\C^m)$ is dense in $D(\mathfrak{a})$.
\end{prop}	

\begin{proof}
We have already noticed that $C_c^\infty(\R^d;\C^m)$ is a subset of $D(\mathfrak{a})$. To prove density, we split the proof into three steps. Preliminarily, we note that 
$\|\f\|_{D(\a)}^2=\|\f\|_{\textup{H}^1_{\a}(\R^d;\C^m)}^2+\| V^{\frac{1}{2}}\f\|_{L^2(\R^d;\C^m)}^2$ for every $\f\in D(\mathfrak a)$.
			
\textit{Step 1.} Here, we prove that the functions of $D(\mathfrak{a})$ with compact support are dense in $D(\mathfrak{a})$. Let $\f\in D(\mathfrak{a})$ and let $\ph\in C_c^\infty(\R^d)$ be a function such that $\mathds{1}_{B(1)}\le\ph\le\mathds{1}_{B(2)}$ and $|\nabla \ph|\leq 2$ on $\R^d$. For $x\in\R^d$ and $n\in\N$, we set $\f_n(x) \coloneqq\ph(x/n)\f(x)$. Since $\ph$ is bounded with bounded gradient, it follows that $\f_n\in D(\mathfrak{a})$. Clearly, $(\f_n)_{n\in\N}$ converges to $\f$ in $L^2(\R^d;\C^m)$ and 
$V^{\frac{1}{2}}(\f_n-\f)$ converges to zero in $L^2(\R^d;\C^m)$. To conclude that $\f_n$ converges to $\f$ in $D(\mathfrak a)$ it suffices to show that the diamagnetic gradient of each component of $\f_n$ converges to the diamagnetic gradient of the corresponding component of $\f$ in $L^2(\R^d;\C^d)$. For this purpose, it suffices to observe that
$(\nabla_{\a}f_{n,j})(x)=\varphi(x/n)(\nabla_{\a} f_j)(x)+n^{-1}\nabla\varphi(x/n)f_j(x)$ for almost every $x\in\R^d$, every $n\in\N$ and $j\in\{1,\ldots,m\}$ so that
\begin{eqnarray*}
\|\nabla_{\a}f_{n,j}-\nabla_{\a}f_j\|_2^2\le 2\int_{\R^d}(1-\varphi(x/n))^2|(\nabla_{\a} f_j)(x)|^2dx+4n^{-2}\|f_j\|_2^2,
\end{eqnarray*}
from which the convergence of $\nabla_{\a}f_{n,j}$ to $\nabla_{\a}f_j$ in $L^2(\R^d;\C^d)$ follows immediately.
			
\textit{Step 2.} Let $\f\in D(\mathfrak{a})$. By Step 1 we can suppose that $\f$ has compact support. Consider a decreasing function $g\in C_c^\infty([0,\infty))$ such that $\mathds{1}_{[0,1]}\le g\le \mathds{1}_{[0,2]}$. 
For $n\in\N$ and $t\in [0,\infty)$, we set $g_n(t)=g(t/n)$ and define $\f_n=\f g_n(\|\f\|)$, which is a bounded function with compact support.
            
By dominated convergence, the sequences $(\f_n)_{n\in\N}$ and
$(V^{\frac{1}{2}}(\f_n-\f))_{n\in\N}$ converge to $\f$ and zero, respectively, in $L^2(\R^d;\C^m)$.

Next, let us fix $j\in\{1,\dots,m\}$ and notice that
$\nabla_{\a} f_{n,j} =  g_n(\|\f\|)\nabla_{\a} f_j + f_j g_n'(\|\f\|)\nabla\|\f\|$ for every $n\in\N$. 
Since 
$|f_j g_n'(\|\f\|)|\le 2\|g'\|_{\infty}$ for every $n\in\N$,
the dominated convergence and Lemma \ref{lem:first:diamagnetic} imply that $\nabla_{\a}f_{n,j}$ converges to
$\nabla_{\a}f_j$ in $L^2(\R^d;\C^d)$ as $n$ tends to $\infty$. We have thus proved that $(\f_n)_{n\in\N}$ converges to $\f$ in $D(\mathfrak a)$ as $n$ tends to $\infty$ and, consequently, bounded functions in $D(\mathfrak{a})$ with compact support are dense in $D(\mathfrak{a})$.
			
\textit{Step 3.} Let $\f\in D(\mathfrak{a})$ be a bounded function with compact support. Since $\nabla_{\a}f_j\in L^2(\R^d;\C^d)$, so is $\a f_j$ for every $j\in\{1,\dots,m\}$, which means that $\f\in \textup{H}^1(\R^d;\C^m)$. Let $\vartheta\in C_c^\infty(\R^d)$ be a positive function such that $\|\vartheta\|_{L^1(\R^d)}=1$ and for every $n\in\N$ we consider the function $\vartheta_{n}$ defined by $\vartheta_n(x)=n^{d}\vartheta(nx)$ for every $x\in\R^d$ and introduce the vector-valued function $\f_{n}=(f_1\star\vartheta_{n},\dots,f_m\star\vartheta_{n})$,
where $\star$ stands for convolution.
Clearly, each function $\f_{n}$ belongs to $C^\infty_c(\R^d;\C^m)$ and converges to $\f$ in $\textup{H}^1(\R^d;\C^m)$. Since 
there exists a compact set $K$ which contains the support of all the functions $\f_n$, it follows that $\f_n$ converges to $\f$ also in $\textup{H}^1_{\a}(\R^d;\C^m)$. 

Finally, $\f_n$ converges to $\f$ pointwise almost everywhere in $\R^d$ and the sequence $(\f_n)_{n\in\N}$ is bounded with ${\rm supp}(\f_n)\subset K$ for every $n\in\N$, again by dominated convergence we deduce that
$V^{\frac{1}{2}}(\f_n-\f)$ converges to zero in $L^2(\R^d;\C^m)$ as $n$ tends to $\infty$.
\end{proof}

\section{The  diamagnetic inequality for semigroups and consequences}
\label{sec:diam:sg}
In this section, our aim is to prove a diamagnetic inequality
and some of its consequences.
For this purpose, we introduce the form $\mathfrak{h}$ associated to the scalar operator $-\Delta+\lambda_V$ in $L^2(\R^d;\C)$, which is defined by
\begin{equation*}
\mathfrak{h}(u,v)=\int_{\R^d}\langle \D u(x),\D v(x)\rangle dx +\int_{\R^d} \lambda_V(x) u(x)\overline{v(x)}dx,\qquad u,v\in	D(\mathfrak{h}),
\end{equation*}
where $D(\mathfrak{h})\coloneqq\{u\in \textup{H}^1(\R^d;\C)\colon \sqrt{\lambda_V}u\in L^2(\R^d;\C)\}$. We recall that $\lambda_V\ge0$ and $\lambda_V\in L^1_{\rm loc}(\R^d)$ by Hypothesis \ref{hyp-operatore}(ii). Since the form is densely defined, accretive, symmetric and closed, the operator $H$, associated to the form $\mathfrak{h}$, is a realization in $L^2(\R^d;\C)$ of the operator $-\Delta+\lambda_V$  and $-H$ generates a positive strongly continuous analytic semigroup of contractions, which we denote by $(e^{-tH})_{t\ge0}$. It is well-known that $(e^{-tH})_{t\ge0}$ extends to a strongly continuous semigroup $(e^{-tH_p})_{t\ge0}$ on $L^p(\R^d;\C)$ ($p\in [1,\infty)$) with infinitesimal generator $-H_p$. The semigroups are consistent and, for every $f\in L^p(\R^d;\C)\cap L^2(\R^d;\C)$ and $\varepsilon>0$, $H_p(\varepsilon+H_p)^{-1}f=-\Delta (\varepsilon+H_p)^{-1}f+\lambda_V (\varepsilon+H_p)^{-1}f\in L^p(\R^d;\C)\cap L^2(\R^d;\C)$.

We can now prove the fundamental diamagnetic inequality.

\begin{theo}
\label{theo-ouh}
For every $\uu\in L^2(\R^d;\C^m)$ the diamagnetic inequality
\begin{equation}\label{eq:diamagnetic:inequality}
\|(e^{-t{\bm A}}\uu)(x)\|\le (e^{-tH}\|\uu\|)(x),\qquad\;\,t>0,\;\,\textit{a.e.}\ x\in\R^d, 
\end{equation}
holds true.
\end{theo}

\begin{proof}
This domination is based on an 
abstract criterion for forms (see \cite[Theorem 2.30]{Ouhabaz}). To apply such a criterion, we need to verify the following properties:
\begin{enumerate}[\rm (i)]
\item
if ${\uu}\in D(\mathfrak{a})$, then $\|\uu\|\in D(\mathfrak{h})$;
\item
if ${\uu}\in D(\mathfrak{a})$ and $f\in D(\mathfrak{h})$ is such that $|f|\le\|\uu\|$, then $|f|\sign{\uu}\in D(\mathfrak{a})$ and
\begin{equation}\label{eq:form:inequality}
		\re\mathfrak{a}({\uu},|f|\sign{\uu}) \ge \mathfrak{h}(\|\uu\|,|f|).
\end{equation}
\end{enumerate}

We begin with property (i). Fix $\uu\in D(\mathfrak{a})$ and recall that $\|\uu\|\in \textup{H}^1(\R^d)$ by Lemma \ref{lem:first:diamagnetic}. In addition, since
\begin{equation*}
\int_{\R^d}\lambda_V\|\uu\|^2 dx \le \int_{\R^d} \langle V\uu,\uu\rangle dx < \infty,
\end{equation*}
the function $\sqrt{\lambda_V}\|\uu\|$ belongs to $L^2(\R^d)$ and we conclude that $\|\uu\|\in D(\mathfrak{h})$.

Next, we prove property (ii). For this purpose, fix ${\uu}\in D(\mathfrak{a})$ and $f\in D(\mathfrak{h})$ such that $|f|\le\|\uu\|$, and set 
$\displaystyle\bm{g}_{\varepsilon}\coloneqq|f|\frac{\uu}{\norm{\uu}+\eps}$ for every $\varepsilon>0$.

Note that $\bm{g}_{\varepsilon}$ converges to $|f|\sign{\uu}$ almost everywhere in $\R^d$ as $\varepsilon$ tends to zero. Moreover, $\|\bm{g}_{\varepsilon}\|\le|f|$ for every $\varepsilon>0$. Hence, $\bm{g}_{\varepsilon}$ converges to $|f|\sign{\uu}$ in $L^2(\R^d;\C^m)$ as $\eps$ tends to $0$.
	
Next, we observe that, for every  $j\in\{1,\dots,m\}$, the weak magnetic gradient of ${g}_{\varepsilon,j}$ is given by
\begin{align*}
\nabla_{\a}g_{\varepsilon,j} 
= & \nabla g_{\varepsilon,j} - i\a  g_{\varepsilon,j}\\ 
=& \frac{u_j}{\|\uu\|+\eps}\D|f| + |f|\left(\frac1{\|\uu\|+\eps}\D_{\a} u_j - \frac1{(\|\uu\|+\eps)^2}u_j\D\|\uu\|\right).
\end{align*}

It can be easily checked that, by dominated convergence $\nabla g_{\varepsilon,j}$ converges in $L^2(\R^d;\C^m)$, so that $|f|{\rm sign}(\bm u)$ belongs to $\textup{H}^1_{\a}(\R^d;\C^m)$ and 
\begin{align*}
\nabla_{\a}(|f|{\rm sign}(\bm u))_j
= \bigg ( \frac{u_j}{\|\uu\|}\D|f| + \frac{|f|}{\|\uu\|}\D_{\a} u_j - \frac{|f|}{\|\uu\|^2}u_j\D\|\uu\|\bigg ){\mathds 1}_{\{\uu\ne \bm{0}\}}, \qquad j\in\{1,\ldots,m\}.
\end{align*}	
To prove that $|f|\sign{\uu}\in D(\mathfrak{a})$ it remains to show that $V^{\frac12}|f|\sign{\uu}$ belongs to $L^2(\R^d;\C^m)$, but this follows immediately observing that
$\|V^{\frac12}|f|{\rm sign}(\uu)\|^2\le \langle V\uu,\uu\rangle$ almost everywhere in $\R^d$.

To complete the proof, we need to show the inequality \eqref{eq:form:inequality}.
For this aim, we fix $({\uu},f)\in D(\mathfrak{a})\times D(\mathfrak{h})$ such that $|f|\le\|\uu\|$ and we start with rewriting the magnetic gradient part
\begin{align*}
&\langle\D_{\a} u_j,\D_{\a}(|f|\sign{\uu})_j\rangle\\
=& \frac{\overline{u_j}}{\|\uu\|}{\mathds 1}_{\{\uu\ne \bf{0}\}}\langle\D_{\a} u_j, \D|f|\rangle + \frac{|f|}{\|\uu\|}{\mathds 1}_{\{\uu\ne \bf{0}\}}\langle\D_{\a} u_j,\D_{\a} u_j\rangle -\frac{|f|}{\|\uu\|^2}\overline{u_j}{\mathds 1}_{\{\uu\ne \bf{0}\}}\langle\D_{\a} u_j, \D\|\uu\|\rangle\\
=&\frac{\overline{u_j}}{\|\uu\|}{\mathds 1}_{\{\uu\ne \bf{0}\}}\langle\D u_j, \D|f|\rangle -i \frac{|u_j|^2}{\|\uu\|}\langle {\a}, \D|f|\rangle + \frac{|f|}{\|\uu\|}{\mathds 1}_{\{\uu\ne \bf{0}\}}\langle\nabla_{\a} u_j,\D_{\a} u_j\rangle\\
&-\frac{|f|}{\|\uu\|^2}\overline{u_j}{\mathds 1}_{\{\uu\ne \bf{0}\}}\langle\D u_j, \D\|\uu\|\rangle+ i\frac{|f|}{\|\uu\|^2}|u_j|^2{\mathds 1}_{\{\uu\ne \bf{0}\}}\langle \bm{a},\D\|\uu\|\rangle.
\end{align*}

From this, taking the real part and summing up from $1$ to $m$, taking \eqref{eq:grad:norma} into account we get
\begin{align*}
&\sum_{j=1}^m\re(\langle\D_{\a} u_j,\D_{\a}(|f|\sign{\uu})_j)\rangle\\
=&\frac{1}{\|\uu\|}\sum_{j=1}^m \langle\re(\overline{u_j}\D u_j), \D|f|\rangle {\mathds 1}_{\{\uu\ne \bf{0}\}}\\
&+\frac{|f|}{\|\uu\|}\bigg( \sum_{j=1}^m\langle\D_{\a} u_j,\D_{\a} u_j\rangle -\frac{1}{\|\uu\|}\sum_{j=1}^m\langle\re(\overline{u_j}\D u_j), \D\|\uu\|\rangle\bigg){\mathds 1}_{\{\uu\ne \bm{0}\}}\\
=&\langle\nabla\|\uu\|,\nabla|f|\rangle 
+\frac{|f|}{\|\uu\|}\bigg( \sum_{j=1}^m\langle\D_{\a} u_j,\D_{\a} u_j\rangle -\langle\D\norm{\bf u}, \D\norm{\bf u}\rangle\bigg)
{\mathds 1}_{\{\uu\ne \bf{0}\}}.
\end{align*}
At this point, from \eqref{eq:diamagnetic:01} and the definition of $\lambda_V$, one obtains
\begin{align}
\re\mathfrak{a}({\uu},|f|\sign{\uu})=&\int_{\R^d}\re\bigg (\sum_{j=1}^m\langle\D_{\a} u_j,\D_{\a}(|f|\sign{\uu})_j\rangle+\langle V{\uu},|f|\sign{\uu}\rangle\bigg )dx\notag
\\
=&\int_{\R^d}\langle\nabla\|\uu\|,\nabla|f|\rangle dx
+\int_{\R^d}\frac{|f|}{\|\uu\|}\langle V\uu,\uu\rangle {\mathds 1}_{\{\uu\ne \bf{0}\}} dx\notag\\
&+\int_{\R^d}\frac{|f|}{\|\uu\|}\bigg( \sum_{j=1}^m|\D_{\a} u_j|^2 -\langle\D\norm{\bf u}, \D\norm{\bf u}\rangle\bigg){\mathds 1}_{\{\uu\ne \bf{0}\}}dx\notag\\
\ge& \int_{\R^d}\langle\nabla\|\uu\|,\nabla|f|\rangle dx+\int_{\R^d}\lambda_V\|\uu\||f|dx = \mathfrak{h}(\|\uu\|,|f|),
\label{turner}
\end{align}
which is \eqref{eq:form:inequality}.
\end{proof}

\begin{rem}
\begin{enumerate}[\rm (i)]
\item 
The proof of Theorem 3.1 also shows that
\begin{equation} 
\|e^{-t\bm{A}}\bm u\|\leq e^{t\Delta}\|\bm u\|,\qquad\;\,t>0,
\label{weak-diamagnetic}
\end{equation}
almost everywhere in $\R^d$ for every $\uu\in L^2(\R^d;\C^m)$, where $(e^{t\Delta})_{t\ge 0}$ denotes the Gauss-Weierstrass semigroup.
\item 
Theorem \ref{theo-ouh}, applied in the case of the vector-valued pure electric Schr\"odinger operator $-\bm\Delta + V$, i.e., when $\a$ identically vanishes in $\R^d$, improves the weaker estimate \cite[Formula (5.44)]{ALLR24}.
\end{enumerate}
\end{rem}

An important consequence of \eqref{eq:diamagnetic:inequality} is the following generation result. We stress that the extension of $(e^{-t\bm{A}})_{t\geq0}$ to a strongly continuous semigroup in $L^p(\R^d;\C^m)$, $p\in[1,\infty)$, is a standard result for $L^\infty$-contractive semigroups (see, for instance, \cite[Section 2.2]{Ouhabaz}). 

\begin{prop}\label{prop:smgr_L1}
The following properties hold.
\begin{enumerate}[{\rm(i)}]
\item 
The semigroup $(e^{-t\bm{A}})_{t\geq0}$ admits a kernel representation, i.e., for every $t>0$ there exists a family of kernels $(p_{hk}(t,\cdot,\cdot))_{1\leq h,k\leq m}\subset L^\infty(\R^d\times \R^d;\C)$ such that for every $\bm u\in L^1(\R^d;\C^m)\cap L^2(\R^d;\C^m)$,
\begin{align}
\label{kernel_representation}
((e^{-t\bm{A}}\bm u)(x))_h=\sum_{k=1}^m\int_{\R^d}p_{hk}(t,x,y)u_k(y)dy, \qquad\;\, h\in\{1,\ldots,m\},
\end{align}
for every $t>0$ and almost every $x\in\R^d$. Moreover, 
\begin{align}
\label{stima_nuclei_gauss}
|p_{hk}(t,x,y)|\leq \frac{1}{(4\pi t)^{\frac d2}}\exp\left(-\frac{|x-y|^2}{4t}\right), \qquad\;\, h,k\in\{1,\ldots,m\},
\end{align}
for every $(t,x,y)\in(0,\infty)\times \R^d\times \R^d$, which implies that \eqref{kernel_representation} extends to every $\bm u\in  L^2(\R^d;\C^m)$.
\item The semigroup $(e^{-t\bm{A}})_{t\geq0}$ extends to a strongly continuous semigroup of contractions $(e^{-t\bm{A}_p})$ on $L^p(\R^d;\C^m)$ for every $p\in[1,\infty)$. All these semigroups are consistent in the sense that they coincide on $L^2(\R^d;\C^m)\cap L^p(\R^d;\C^m)$. In particular, for every $\bm u\in L^p(\R^d;\C^m)\cap L^2(\R^d;\C^m)$, it holds that
\begin{align*}
\|(e^{-t\bm{A}_p}\bm u)(x)\|\leq (e^{-tH_p}\|\bm u\|)(x), \qquad\;\, \textit{a.e.}~x\in\R^d.    
\end{align*}
\item 
For every $p\in[1,\infty)$ and every $\bm f\in L^p(\R^d;\C^m)$, the map $t\mapsto e^{-t\bm{A}_p}\bm f$ belongs to $ C((0,\infty);L^\infty(\R^d;\C^m))$.
\item 
For every $p\in[1,\infty)$, the set $(1+\bm{A}_p)^{-1}(C^{\infty}_c(\R^d;\C^m))$ is dense in $L^p(\R^d;\C^m)$ and in $D(\bm{A}_p)$, endowed with the graph norm.
\item 
If $\bm f\in (1+\bm{A}_p)^{-1}(C^{\infty}_c(\R^d;\C^m))$ for some $p\in[1,\infty)$, then $\bm{A}_p\bm f=-\bm\Delta_\a\bm f +V\bm f\in L^p(\R^d;\C^m)$, where $\bm\Delta_\a\bm f +V\bm f$ is meant in the sense of distributions.
\end{enumerate}
\end{prop}
\begin{proof}
(i) can be proved arguing as in the proof of \cite[Proposition 4.3]{ALLR-2}, by simply noticing that $(e^{-t\bm{A}})_{t\geq0}$ is ultracontractive by \eqref{weak-diamagnetic}. This implies that we can apply Dunford-Pettis theorem to infer that, for every $t>0$, there exists a family of kernels $(\widetilde p_{hk}(t,\cdot,\cdot))_{1\leq h,k\leq m}\in L^\infty(\R^d\times\R^d)$ which satisfies \eqref{kernel_representation}.
Further, still from \eqref{weak-diamagnetic} (which provides Gaussian estimates for $\widetilde p_{hk}(t,x,\cdot)$ for every $t>0$ and almost every $x\in\R^d$) and the proof of \cite[Theorems 3.3.4 \& 3.3.5]{Meyer-Nieberg}, we infer that, for every $t>0$, estimate \eqref{stima_nuclei_gauss} holds true, with $\widetilde p_{hk}$ instead of $p_{hk}$, for almost every $(x,y)\in\R^d\times \R^d$. To get the validity of estimate \eqref{stima_nuclei_gauss} for every $(x,y)\in\R^d\times \R^d$ and every $t>0$, we modify the kernels as follows: for every $h,k\in\{1,\ldots,m\}$ and every $t>0$, we set
\begin{align*}
p_{hk}(t,x,y)
=\begin{cases}
\widetilde p_{hk}(t,x,y), & \textit{if }\displaystyle|\widetilde p_{hk}(t,x,y)|\leq \frac{1}{(4\pi t)^\frac d2}\exp\left(-\frac{|x-y|^2}{4t}\right), \\[2mm]
0, & \textit{otherwise}.
\end{cases}
\end{align*}
Clearly, for every $t>0$, the set $\{y\in\R^d:\widetilde p_{hk}(t,x,y)\neq p_{hk}(t,x,y), \ h,k\in\{1,\ldots,m\}\}$ has null Lebesgue measure in $\R^d$. Therefore, for almost every $x\in \R^d$ and every $h\in\{1,\ldots,m\}$,
\begin{align*}
((e^{-t\bm{A}}\bm f)(x))_h
= \sum_{k=1}^m\int_{\R^d}\widetilde p_{hk}(t,x,y)f_k(y)dy
=\sum_{k=1}^m\int_{\R^d}p_{hk}(t,x,y)f_k(y)dy
\end{align*}
for every $\bm f\in L^1(\R^d;\C^m) \cap L^2(\R^d;\C^m)$, so that also \eqref{kernel_representation} is satisfied.
\smallskip

(ii) From \eqref{stima_nuclei_gauss} and the contractivity of the
Gauss-Weierstrass semigroup in $L^p(\R^d;\C^m)$ for every $p\in [1,\infty)$, it follows that each operator $e^{-t\bm{A}}$ extends to a contraction on $L^p(\R^d;\C^m)$. The semigroup property is inherited by density in $L^p(\R^d;\C^m)$ from the corresponding property in $L^2(\R^d;\C^m)$. Hence, $(e^{-t\bm{A}})_{t\ge 0}$ extends to the scale $L^p(\R^d;\C^m)$ ($p\in [1,\infty)$), and, denoting these semigroups by $(e^{-t\bm{A}_p})_{t\ge 0}$, it is easy to check that they are consistent.

We notice that $(e^{-t\bm{A}_p})_{t\ge 0}$ is a strongly continuous semigroup (of contractions) in $L^p(\R^d;\C^m)$ for every $p$ as above. In fact, the case $p\in[2,\infty)$ follows from a density argument, the case $p\in(1,2]$ can be proved by duality and the case $p=1$ follows as in the scalar case, see for instance \cite[p. 56]{Ouhabaz}. 

(iii) Fix $\bm f\in C^{\infty}_c(\R^d;\C^m)$ and notice that, for every $t_0\in(0,\infty)$, every $t\in(t_0,\infty)$ and almost every $x\in\R^d$, from property (i) it follows that
\begin{align*}
\|(e^{-t\bm{A}_p}\bm f)(x)-(e^{-t_0\bm{A}_p}\bm f)(x)\|
= & \|(e^{-t_0\bm{A}_p}(e^{-(t-t_0)\bm{A}_p}\bm f-\bm f))(x)\|\\
\le & (e^{t_0\Delta}\|e^{-(t-t_0)\bm{A}_p}\bm f-\bm f\|)(x),
\end{align*}
so that 
\begin{equation*}
\|e^{-t\bm{A}_p}\bm f-e^{-t_0\bm{A}_p}\bm f\|_{\infty}
\le  \|e^{t_0\Delta}\|e^{-(t-t_0)\bm{A}_p}\bm f-\bm f\|\|_{\infty}
\le (4\pi t_0)^{-\frac{d}{2}}\|e^{-(t-t_0)A_1}\bm f-\bm f\|_1.
\end{equation*}
This estimate implies that $e^{-t\bm{A}_p}\f$ converges in $L^{\infty}(\R^d;\C^m)$ to $e^{-t_0\bm{A}_p}\f$ as $t$ tends to $t_0$ from the right.
With a slight modification in the previous argument, it can be also proved that $\|e^{-t\bm{A}_p}\f-e^{-t_0\bm{A}_p}\f\|_{\infty}$ tends to zero as $t$ tends to $t_0$ from the left.

The density of $C^{\infty}_c(\R^d;\C^m)$ in $L^p(\R^d;\C^m)$ together with the contractiveness of the operators $e^{-t\bm{A}_p}$ allow to extend the previous result to every $\f\in L^p(\R^d;\C^m)$.

(iv) Fix $p\in [1,\infty)$ and $\f\in D(\bm{A}_p)$. Then, there exists
$\g\in L^p(\R^d;\C^m)$ such that $\f=(1+\bm{A}_p)^{-1}\g$. Approximating $\g$ with a sequence $(\g_n)_{n\in\N}\subset C^{\infty}_c(\R^d;\C^m)$, which converges to $\g$ in $L^p(\R^d;\C^m)$, it can be easily proved that the sequence $(\f_n)_{n\in\N}$, defined by 
$\f_n=(1+\bm{A}_p)^{-1}\bm{g}_n$ for every $n\in\N$, converges to $\f$ in $(D(\bm{A}_p),\|\cdot\|_{D(\bm{A}_p)})$. Finally, recalling that $D(\bm{A}_p)$ is dense in $L^p(\R^d;\C^m)$, we conclude that 
the set $(1+\bm{A}_p)^{-1}(C^{\infty}_c(\R^d;\C^m))$ is dense also in $L^p(\R^d;\C^m)$.

(v) If $\bm f=(1+\bm{A}_p)^{-1}\bm{g}$ for some $\bm{g}\in C^{\infty}_c(\R^d;\C^m)$ for some $p\in [1,\infty)$, then, from the consistency of the semigroups, $\bm f=(1+\bm{A}_p)^{-1}\bm g=(1+\bm{A})^{-1}\bm g$ for some $\bm g\in L^p(\R^d;\C^m)\cap L^2(\R^d;\C^m)$. From the definition of $\bm{A}$, we infer that $\bm A\bm f=-\bm\Delta_\a\bm f+V\bm f$ in the sense in distributions. Therefore,
\begin{equation*}
\bm{A}_p\bm f=\bm g-\bm f= \bm A\bm f=-\bm\Delta_\a\bm f+V\bm f.    \qedhere
\end{equation*}
\end{proof}

An application of this result, together with \cite[Theorem 3.6]{ALLR-2}, yields the boundedness of the $\textup{H}^\infty$-calculus for $\bm{A}_p$ on $L^p(\R^d;\C^m)$.

\begin{cor}
\label{coro:Hinfinity}
The operator $\bm{A}_p$ admits bounded $\textup{H}^\infty$-calculus on $L^p(\R^d;\C^m)$ for every $p\in(1,\infty)$.
\end{cor}

\begin{proof}
The case $p=2$ is obvious since $\bm A$ is a non-negative self-adjoint operator (see for instance \cite[Corollary 7.1.6]{haase2006}). 
The general case follows from \cite[Theorem 3.6]{ALLR-2}, whose assumptions are fulfilled thanks to Proposition \ref{prop:prop:a} and Proposition \ref{prop:smgr_L1}(i). 
\end{proof}

For further use, we need to consider kernels which are measurable with respect to the triple $(t,x,y)\in(0,\infty)\times \R^d\times \R^d$. 

\begin{prop}
\label{proposition:measurability}
There exists a family of kernels $\{q_{hk}: h,k\in\{1,\ldots,m\}\}$ such that, for 
every $\f\in L^p(\R^d;\C^m)$ and almost every $t\in (0,\infty)$,
\begin{align}
e^{-t\bm{A}_p}\f=\sum_{k=1}^m\int_{\R^d}q_{hk}(t,\cdot,y)f_k(y)dy,\qquad\;\,\textit{a.e. in }\R^d.
\label{semigruppo-ker}
\end{align}
Moreover, each function $q_{hk}$ is measurable in $(0,\infty)\times\R^d\times\R^d$ with values in $\C$ and 
\begin{align}
\label{stima_exp_nuclei_q}
|q_{hk}(t,x,y)|\leq \frac{1}{(4\pi t)^{\frac d2}}\exp\bigg(-\frac{|x-y|^2}{4t}\bigg ),\qquad\;\,\textit{a.e. in } (0,\infty)\times\R^d\times\R^d.   
\end{align}
\end{prop}

\begin{proof}
To begin with, we recall that for every function $v\in C_b([0,\infty);L^p(\R^d;\C))$, with some $p\in [1,\infty)$, there exists a measurable function $\widetilde v:(0,\infty)\times\R^d\to\C$, such that $\widetilde v\in L^p((0,T)\times\R^d;\C)$ for every $T>0$ and for almost every $t\in (0,\infty)$, the equality $v(t)=\widetilde v(t,\cdot)$ holds almost everywhere in $\R^d$. This is a classical result from measure theory. For the sake of completeness, we provide a detailed proof in the Appendix (see Proposition \ref{lemma:rappr_misurabile}).

The previous result applies to the function $t\mapsto (e^{-t\bm{A}_p}(f{\bm e}_k))_h$ for every 
$f\in L^p(\R^d;\C)$ and every $h,k\in\{1,\ldots,m\}$. We denote by $\widetilde u_{hk}^{f,p}$ the corresponding function in $L^p((0,T)\times\R^d;\C)$, for every $T>0$, such that
$(e^{-t \bm{A}_p}(f\bm e_k))_h=\widetilde u_{hk}^{f,p}(t,\cdot)$ almost everywhere in $\R^d$, for almost every $t\in (0,\infty)$.

Fix $\varepsilon\in(0,T)$ and $h,k\in\{1,\ldots,m\}$, and consider the set
$F$ of all finite linear combinations of functions of the form $\mathds 1_B\otimes v$ such that $B\in\mathscr{L}((\varepsilon,T))$ and $v\in L^1(\R^d;\C)$. On $F$, which is dense in $L^1((\varepsilon,T)\times\R^d;\C)$, we introduce the linear operator $\mathscr B_{hk}$, defined by 
\begin{align*}
\mathscr B_{hk}(u)
= \int_{\varepsilon}^T(e^{-t\bm{A}_1}u(t,\cdot)\bm e_k)_hdt,\qquad\;\,u\in F.
\end{align*}
From Proposition \ref{prop:smgr_L1}(iii), it follows that, for every $v\in L^1(\R^d)$, the function  
$t\mapsto (e^{-tA_1}(v\bm e_k))$ belongs to $C((0,\infty);L^\infty(\R^d;\C^m))$ and, hence, to $L^1((\varepsilon,T);L^\infty(\R^d;\C^m))$ for every $\varepsilon>0$. 

Let $u=\sum_{i=1}^rc_i\mathds{1}_{B_i}\otimes v_i$, with $B_i\cap B_j=\varnothing$ if $i\neq j$, which is not restrictive, and $c_1,\ldots,c_r\in\C$.
Formulae \eqref{kernel_representation} and \eqref{stima_nuclei_gauss} show that, for every $t\in (\varepsilon,T)$, 
\begin{align*}
\left|((e^{-t\bm{A}_1}u(t,\cdot)\bm e_k)(x))_h\right|
= & \bigg |\sum_{i=1}^r{c_i}\mathds 1_{B_i}(t)\int_{\R^d}p_{hk}(t,x,y)v_i(y)dy\bigg | \\
\leq &\frac{1}{(4\pi t)^{\frac d2}}
\sum_{i=1}^r|c_i|\|v_i\|_1\mathds 1_{B_i}(t)
\end{align*}
for almost every $x\in\R^d$. Hence,
it follows that $\mathscr B_{hk}(u)\in L^{\infty}(\R^d)$ and
\begin{align}
\label{stima_bhk}
\|\mathscr B_{hk}(u)\|_\infty 
\le\frac{1}{(4\pi \varepsilon)^\frac d2}\|u\|_{L^1((\varepsilon,T)\times \R^d)}.
\end{align}
By density, \eqref{stima_bhk} can be extended to every $u\in L^1((\varepsilon,T)\times \R^d;\C)$.

Let us now introduce the operator $\Phi\colon\widetilde F\to \R$, defined by
\begin{align*}
\Phi(u)=\int_{\R^d}{\mathscr B}_{hk}(u(\cdot,x,\cdot))dx, \qquad u\in \widetilde F,   
\end{align*}
where $\widetilde F$ is the set of all finite linear combinations of functions of the form 
$\mathds 1_B\otimes \mathds 1_{E}\otimes v$ where $B\in\mathscr{L}((\varepsilon,T))$, $E\in\mathscr{L}(\R^d)$ and $v\in L^1(\R^d;\C)$.

From the definition of $\Phi$ and the estimate on the $L^{\infty}$-norm of $\mathscr B_{hk}(u)$, it follows that, if 
\begin{align*}
u(t,x,y)=\sum_{i=1}^rc_i\mathds 1_{B_i}(t)\mathds 1_{E_i}(x)v_i(y), \qquad\;\, (t,x,y)\in (\varepsilon,T)\times \R^d\times \R^d,   
\end{align*}
with $(B_i\times E_i)\cap (B_j\times E_j)=\varnothing$ if $i\neq j $ (which is not restrictive), then 
\begin{align*}
|\Phi(u)|
=\bigg |\sum_{i=1}^rc_i\int_{\R^d}\mathscr{B}_{hk}(\mathds 1_{B_i}\otimes v_i)\mathds 1_{E_i}dx\bigg |
\le \frac{1}{(4\pi\varepsilon)^\frac d2}\|u\|_{L^1((\varepsilon,T)\times \R^d\times \R^d;\C)}.
\end{align*}
Hence, $\Phi$ extends to a linear continuous functional on $L^1((\varepsilon,T)\times \R^d\times \R^d;\C)$, and so there exists $q_{hk}\in L^\infty((\varepsilon,T)\times \R^d\times \R^d;\C)$ such that
\begin{align*}
\Phi(u)=\int_{(\varepsilon,T)\times\R^d\times \R^d}q_{hk}(t,x,y)u(t,x,y)dtdxdy.    
\end{align*}
In particular, if $u=\mathds 1_{B}\otimes\mathds 1_{E}\otimes \mathds{1}_F$, with $B\in \mathscr {\mathcal B}((\varepsilon,T))$, $E,F\in \mathscr {\mathcal B}(\R^d)$, then from \eqref{ug_int_punt} it follows that
\begin{align*}
\int_{B\times E\times F}q_{hk}(t,x,y)dtdxdy
=& \int_E\bigg (\int_B(e^{-t\bm{A}_1}(\mathds{1}_F\bm e_k))_hdt\bigg )(x)dx \\
= & \int_Edx\int_B\widetilde u_{hk}^{\mathds{1}_F}(t,x)dt\\
=& \int_Bdt\int_E\widetilde u_{hk}^{\mathds{1}_F}(t,x)dx \\
= & \int_Bdt\int_E(e^{-t\bm{A}_1}(\mathds{1}_F\bm e_k))_h(x)dx.
\end{align*}
From this chain of equalities, \eqref{kernel_representation} and \eqref{stima_nuclei_gauss}, we infer that
\begin{align*}
\bigg |\int_{B\times E\times F}q_{hk}(t,x,y)dtdxdy\bigg |
\leq & \int_Bdt\int_Edx \int_{\R^d}|p_{hk}(t,x,y)|\mathds 1_F(y)dy \\
\le & \int_{B\times E\times F}\frac{1}{(4\pi t)^{\frac d2}}\exp\left(-\frac{|x-y|^2}{4t}\right)dtdxdy.
\end{align*}
The arbitrariness of $B$, $E$ and $F$ implies that
\begin{align*}
|q_{hk}(t,x,y)|\leq \frac{1}{(4\pi t)^{\frac d2}}\exp\left(-\frac{|x-y|^2}{4t}\right)  
\end{align*}
for almost every $(t,x,y)\in(\varepsilon,T)\times \R^d\times \R^d$. Estimate \eqref{stima_exp_nuclei_q} follows immediately.

To prove \eqref{semigruppo-ker} we fix $p\in [1,\infty)$ and notice that, from the consistency of the semigroups $(e^{-t\bm A_q})_{t\ge 0}$, for $q\in [1,\infty)$,
Proposition \ref{prop:smgr_L1}(iii) and the second part of Lemma \ref{lemma:rappr_misurabile}, it follows that, if $u=\mathds 1_{B}\otimes\mathds 1_{E}\otimes v$, with $B\in \mathscr{L}((\varepsilon,T))$, $E\in \mathscr{L}(\R^d)$ and $v\in L^1(\R^d;\C)\cap L^p(\R^d;\C)$, then
\begin{align*}
\int_{B\times E}dtdx\int_{\R^d}q_{hk}(t,x,y)v(y)dy
= & \int_E\bigg (\int_B(e^{-t\bm{A}_1}(v \bm e _k))_hdt\bigg )(x)dx\\
= & \int_E\bigg (\int_B(e^{-t\bm{A}_p}(v \bm e _k))_hdt\bigg )(x)dx.
\end{align*}

If $p\in (1,\infty)$, then we extend the previous equality to every $v\in L^p(\R^d)$, taking a sequence $(v_n)\subset L^1(\R^d;\C)\cap L^p(\R^d;\C)$ converging to $v$ in $L^p(\R^d;\C)$.
Writing 
\begin{align}
\label{rapp_q_hk-n}
\int_{B\times E}dtdx\int_{\R^d}q_{hk}(t,x,y)v_n(y)dy
=\int_E\bigg (\int_B(e^{-t\bm{A}_p}(v_n \bm e _k))_hdt\bigg )(x)dx
\end{align}
for every $n\in\N$ and letting $n$ tend to $\infty$, the previous formula follows with $v$ instead of $v_n$. Note that the integral on the left-hand side of \eqref{rapp_q_hk-n} converges by dominated convergence thanks to \eqref{stima_exp_nuclei_q}, whereas the right-hand side converges due to the strong continuity of the semigroup $(e^{-t\bm{A}_p})_{t\ge 0}$. It thus follows that
\begin{align*}
\int_{B\times E}\!dtdx\int_{\R^d}q_{hk}(t,x,y)v(y)dy
=&\int_E\bigg (\int_B(e^{-t\bm{A}_p}(v \bm e _k))_hdt\bigg )(x)dx
=  \int_{B\times E}\widetilde u_{hk}^{v,p}(t,x)dtdx
\end{align*}
for every $p\in [1,\infty)$.
Again, the arbitrariness of $B$ and $E$ and $\varepsilon$ yields \eqref{semigruppo-ker}. 
\end{proof}

\section{Maximal inequalities in \texorpdfstring{$L^p(\R^d;\C^m)$}{Lp(Rd;Cm)}}
\label{sec:max_ineq_Lp}

In this section, we prove maximal inequalities for the functions $\bm\Delta_\a \bm u$ and $V\bm u$, when $\bm u\in D(\bm{A}_p)$ under the following additional assumption on the potential $V$, which is assumed throughout the section. 

\begin{hyp}
\label{hyp:avl_potenziale}
The eigenvalues of $V$ are comparable with each other, i.e., if we denote by $\Lambda_V(x)$ the largest eigenvalue of $V(x)$, then there exists a constant $\kappa\geq 1$ such that, for almost every $x\in\R^d$,
\begin{equation*}
\lambda_V(x)\le\Lambda_V(x)\le \kappa\lambda_V(x).
\end{equation*}    
\end{hyp}

To prove the maximal inequalities and avoid some technical difficulties, we use an approximation argument, approximating the  vector-valued potential $V$ with a family of bounded vector-valued potential $V_n$ ($n\in\N$) defined as follows. For every $n\in\N$ and every $x\in\R^d$, we set $F_n\coloneqq\{x\in \R^d:n>\kappa\lambda_V(x)\}$, where $\kappa$ is the positive constant in Hypothesis \ref{hyp:avl_potenziale}, and 
\begin{align*}
V_n(x)\coloneqq{\mathds 1}_{F_n}(x) V(x)+(1-{\mathds 1}_{F_n}(x))(\lambda_V(x)\wedge n)\operatorname{Id}_{\R^m}.    
\end{align*}

Clearly, for every $n\in\N$ the matrix-valued function $V_n$ is bounded and
\begin{equation}
\label{eq:Vn:ge:forma}
( \lambda_V(x)\wedge n)\|\xi\|^2 \le\langle V_n(x)\xi,\xi\rangle\le \kappa(\lambda_V(x)\wedge n)\|\xi\|^2
\end{equation}
for almost every $x\in\R^d$, every $\xi\in\C^m$ and every $n\in\N$.

For every $n\in\N$, we also consider the realization $H_{n,p}$ in $L^p(\R^d;\C)$ of the formal operator $-\Delta+(\lambda_V\wedge n)$, with domain
$D(H_{n,p})=W^{2,p}(\R^d;\C)$, if $p\in (1,\infty)$, and $D(H_{n,1})=\{u\in L^1(\R^d;\C): \Delta u\in L^1(\R^d;\C)\}$, if $p=1$.
By Lemma \ref{lemma:operatori_scalari_n}, each operator $-H_{n,p}$ generates a positive strongly continuous analytic semigroup $(e^{-tH_{n,p}})_{t\ge 0}$ in $L^p(\R^d;\C)$ and the semigroups are consistent.

\begin{theo}
For every $\uu\in D(\bm{A}_1)$ it holds that
\begin{equation}
\label{eq:maximal_ineq:L1_nuovo}
\norm{\bm\Delta_{\a} \uu}_1\leq (\kappa+1)\|\bm{A}_1\bm u\|_1, \qquad  \norm{V\uu}_1 \le \kappa \norm{\bm{A}_1 \uu}_1,
\end{equation}
where $\kappa$ is the constant in Hypothesis $\ref{hyp:avl_potenziale}$.
If $\lambda_V\in B_q$ for some $ q\in (1,\infty)$, then there exists $\delta>0$, depending on $[\lambda_V]_{B_q}$, such that
\begin{equation}\label{eq:maximal_ineq:Lp_nuovo}
\norm{\bm\Delta_{\a} \uu}_p\leq (\kappa c_p'+1)\|\bm{A}_p\bm u\|_p, \qquad\;\,  \norm{V\uu}_p \le \kappa c_p'\norm{\bm{A}_p \uu}_p,\quad\;\,\uu\in D(\bm{A}_p),
\end{equation}
for every $p\in [1,q+\delta]$ and some positive constant $c_p'$. If  $\lambda_V\in B_{\infty}$, then estimate \eqref{eq:maximal_ineq:Lp_nuovo} holds true for every $p\in [1,\infty)$. 
\end{theo}

\begin{proof}
We use an approximation procedure which is not really needed in the case $q\in [2,\infty)$, $p\in [q',q+\delta)$ and in the case $q=\infty$.

It suffices to prove the claim when $q\in (1,\infty)$ since $B_{\infty}\subset B_q$ for every $q\in (1,\infty)$. 

Fix $q\in (1,\infty)$. From \cite{gehring}, it follows that there exists $\delta>0$, which only depends on $[\lambda_V]_{B_q}$, such that $\lambda_V\in B_p$ for every $p\in [1,q+\delta]$.

To begin with, we notice that
\begin{equation}\label{eq:diamagnetic:inequality_n}
\|(e^{-t\bm{A}_p}\uu)(x)\|\le (e^{-tH_{n,p}}\|\uu\|)(x),\qquad\;\,t>0,\;\,\textit{a.e.\ } x\in\R^d, 
\end{equation}
for every $\bm u\in C^{\infty}_c(\R^d;\C^m)$ and every
$p\in [1,\infty)$. The previous diamagnetic inequality follows from observing that the operator $H_{n,2}$ is associated with the form
$\mathfrak{h}_n:\textup{H}^1(\R^d;\C)\times \textup{H}^1(\R^d;\C)\to\C$, defined by 
\begin{equation*}
\mathfrak{h}_n(u,v)=\int_{\R^d}\langle \D u(x),\D v(x)\rangle dx +\int_{\R^d} (\lambda_V(x)\wedge n) u(x)\overline{v(x)}dx,\qquad\;\,
u,v\in	\textup{H}^1(\R^d;\C).
\end{equation*}
Since 
$\mathfrak{a}(\uu,|f|{\rm sign}(\uu))\ge \mathfrak{h}_n(\|\uu\|,|f|)$
for every $\uu\in D(\mathfrak{a})$ and every $f\in \textup{H}^1(\R^d;\C)$ (see \eqref{turner} for further details), the claim follows from \cite[Theorem 2.30]{Ouhabaz}.

By taking estimates \eqref{eq:Vn:ge:forma}, \eqref{eq:diamagnetic:inequality_n} together with the positivity of the semigroup $(e^{-tH_{n,p}})_{t\ge 0}$ into account, it follows that
\begin{align}
\|V_n(\varepsilon+\bm{A}_p)^{-1}\bm v\|_p
\le & \kappa\|(\lambda_V\wedge n)(\varepsilon+H_{n,p})^{-1}\|\bm v\|\|_p \notag \\ 
\le & \kappa\|((\lambda_V\wedge n)+\varepsilon) (\varepsilon+H_{n,p})^{-1}\|\bm v\|\|_p
\label{stima_vett_scal_Lp_nuovo}
\end{align}
for every ${\bm v}\in C^{\infty}_c(\R^d;\C^m)$ and every $\varepsilon>0$. 
We can thus apply Lemma \ref{lemma:operatori_scalari_n} and conclude that 
\begin{align}\label{eq:max:in:l+e_nuovo}
\|((\lambda_V\wedge n)+\eps)(\varepsilon+H_{n,p})^{-1}\|\bm v\|\|_{p}
\le c_p'\|\bm v\|_p
\end{align}
for a positive constant $c_p'$, which depends on $p$, $d$ and $[(\lambda_V\wedge n)+\varepsilon]_{B_p}$ but is independent of $\bm v$. Moreover, Lemma \ref{lemma:Bp_truncate}(iii) states that $[(\lambda_V\wedge n)+\varepsilon]_{B_p}$ can be bounded uniformly with respect to $n\in\N$ and $\varepsilon>0$, which implies that we can choose the constant $c_p'$ independent of $n$ and $\varepsilon$.

Substituting estimate \eqref{eq:max:in:l+e_nuovo} into \eqref{stima_vett_scal_Lp_nuovo} gives
$\|V_n(\varepsilon+\bm{A}_p)^{-1}\bm v\|_p\leq \kappa c_p'\|\bm v\|_p$. Letting $n$ tend to $\infty$ and then using a density argument, we infer that $\|V(\varepsilon+\bm{A}_p)^{-1}\bm v\|_p\leq \kappa c_p'\|\bm v\|_p$ for every $\bm v\in L^p(\R^d;\C^m)$ or, equivalently,
\begin{equation}\label{eq:V:le:L+eps_nuovo}
\|V\uu\|_p\le \kappa c_p'\|(\bm{A}_p+\eps)\uu\|_p,\qquad\;\, \uu\in D(\bm{A}_p).
\end{equation}
Letting $\eps$ tend to $0^+$ in \eqref{eq:V:le:L+eps_nuovo}, we conclude that
\begin{eqnarray*}
\|V\uu\|_p\le \kappa c_p' \|\bm{A}_p\uu\|_p,\qquad\;\, \uu\in D(\bm{A}_p).
\end{eqnarray*}

Next, we prove the maximal inequality for the magnetic Laplace operator. For this purpose, we recall that $\bm \Delta_{\a}\bm f=V\bm f-\bm{A}_p\bm f$ for every $\f\in (1+\bm{A}_p)^{-1}(C^\infty_c(\R^d;\C^m))$ (see Proposition \ref{prop:smgr_L1}(v)), so that
$\bm \Delta_{\a}\bm f\in L^p(\R^d;\C^m)$ and
\begin{align*}
\|\bm\Delta_\a\bm f\|_p\leq (1+\kappa c_p')\|\bm{A}_p\bm f\|_p.
\end{align*}
The density of the set $(1+\bm{A}_p)^{-1}(C^\infty_c(\R^d;\C^m))$ in $D(\bm{A}_p)$ (see Proposition \ref{prop:smgr_L1}(iv)) allows us to complete the proof of \eqref{eq:maximal_ineq:Lp_nuovo}.

The proof of estimate \eqref{eq:maximal_ineq:L1_nuovo} can be obtained in the same way. The only difference lies in the proof of \eqref{eq:max:in:l+e_nuovo}, which now follows from \eqref{max_ineq_1_n} and gives $c'_1=1$.
\end{proof}

\section{First-order Riesz transform}
\label{sec:RT}
Since the operator $\bm{A}$ is maximally accretive and self-adjoint on $L^2(\R^d;\C^m)$, it has a unique $m$-accretive self-adjoint square root, denoted as $\bm{A}^{\frac12}$ (see, e.g.,  \cite[Chapter V.3.11]{kato:book}). In addition,
\begin{equation*}
\mathfrak{a} (\uu,{\bm v}) = \langle \bm{A}^{\frac12}\uu,\bm{A}^{\frac12}{\bm v}\rangle_{L^2(\R^d;\C^m)}
\end{equation*}
for every $\uu,{\bm v}\in D(\bm{A}^{\frac12})= D(\mathfrak{a})$. Moreover, recalling that $V$ is nonnegative definite, it follows that
\begin{align}
\|\nabla_{\a}\uu\|^2_2 &= \sum_{j=1}^m\int_{\R^d}\langle \nabla_{\a} u_j(x),\nabla_{\a} u_j(x)\rangle dx\notag\\
&\le \sum_{j=1}^m\int_{\R^d}\langle \nabla_{\a} u_j(x),\nabla_{\a} u_j(x)\rangle dx+\int_{\R^d}\langle V(x) \bm{u}(x),\bm{u}(x) \rangle dx\notag \\
&= \mathfrak{a} (\uu,\uu)=\|\bm{A}^{\frac12}\uu\|^2_2
\label{eq:nabla:le:square:root}
\end{align}
for every $\uu\in D(\mathfrak{a})$. We start with some observations. 

\begin{rem}
\label{rem:Lp:square:root:scalar}
\begin{enumerate}[\rm (i)]
\item
The arguments above, with $V=0$, apply to the realization $-\bm{L}$ of the magnetic Laplacian $\bm{\Delta_{\a}}$ in $L^2(\R^d;\C^m)$, with domain $\{\uu\in \textup{H}^1_{\a}(\R^d;\C^m): {\bm\Delta}_{\a}\uu\in L^2(\R^d;\C^m)\}$, where $\bm{\Delta}_{\a}\uu$ is understood in the distributional sense. Since $\bm{L}$  acts componentwise, so does its square root. Indeed, the $m$-accretive operator defined by ${\rm diag}(L^{\frac{1}{2}}u_1,\dots,L^{\frac{1}{2}}u_m)$ for every $\uu\in \textup{H}_{\a}^1(\R^d;\C^m)$, where $L$ is the realization in $L^2(\R^d;\C)$, via form methods, of the magnetic Laplacian,
is a square root of $\bm{L}$ and again \cite[Chapter V.3.11]{kato:book} implies that $\textup{H}_{\a}^1(\R^d;\C^m)\ni \bm u\mapsto{\rm diag}(L^{\frac{1}{2}}u_1,\dots,L^{\frac{1}{2}}u_m)$ is the unique $m$-accretive square root of the magnetic Laplacian. 
\item 
In the scalar case, it is well known that the operator 
$\nabla_{\bm a} L^{-\frac{1}{2}}$, initially defined on the range of $L^{\frac{1}{2}}$, extends to a contraction on $L^2(\mathbb{R}^d;\mathbb{C})$. 
From property (i), it follows that the operator 
$\nabla_{\bm a}\bm{L}^{-\frac{1}{2}}$ acts componentwise and is a contraction 
on $L^2(\mathbb{R}^d;\mathbb{C}^m)$.
\item
It follows from the definition that 
$\|\bm{L}^{\frac{1}{2}}\uu\|_2 = \|\nabla_{\bm a}\uu\|_2$
for every $\uu \in \textup{H}_{\a}^1(\mathbb{R}^d;\mathbb{C}^m)$. 
Combined with \eqref{eq:nabla:le:square:root}, this yields
$\|\bm{L}^{\frac{1}{2}}\uu\|_2 \le \|\bm{A}^{\frac{1}{2}}\uu\|_2$
for every $\uu \in D(\mathfrak{a})$.
\item
Property (iii) and estimate \eqref{eq:diamagnetic:01} show that 
$\bm{L}^{\frac{1}{2}}$ is one-to-one. Indeed, if 
$\bm{L}^{\frac{1}{2}}\bm{u} = \bm{0}$ for $\bm{u}\in \textup{H}_{\a}^1(\mathbb{R}^d;\C^m)$, then 
$\nabla \|\bm{u}\| = \bm{0}$ almost everywhere in $\mathbb{R}^d$, which implies $\bm{u} = \bm{0}$ almost everywhere in $\mathbb{R}^d$. 
This observation allows us to define the operator 
$\bm{A}^{-\frac{1}{2}}$ on the range of $\bm{A}^{\frac{1}{2}}$. Since the range of $\bm{A}^{-\frac{1}{2}}$ is $D(\bm{A}^{\frac{1}{2}}) \subset \textup{H}_{\a}^1(\mathbb{R}^d;\C^m)$, the operator $\nabla_{\bm a}\bm{A}^{-\frac{1}{2}}$ is well defined on $D(\bm{A}^{-\frac12})$ and is commonly referred to as the {\it Riesz transform}. 
In the same way, one can show that the operator 
$\bm{L}^{\frac12}\bm{A}^{-\frac{1}{2}}$ is also well defined on $D(\bm{A}^{-\frac{1}{2}})$.
\end{enumerate}
\end{rem}

For every $n\in\N$ and $\alpha\in (0,1)$, the function $\psi_n$, defined by $\psi_n(z)=(n^{-1}+z)^{\alpha-1}$ is bounded and holomorphic in the sector $\{z\in\C\setminus \{0\}\colon |\arg(z) | < \theta \}$ for every $\theta\in(0,\pi)$. Therefore, by functional calculus we can define the operator $(n^{-1} + \bm{A})^{\alpha-1} = \psi_n(\bm{A})$ in $L^2(\R^d;\C^m)$ and write
\begin{equation}
\label{eq:rep:e+H:-1/2}
(n^{-1} + \bm{A})^{-\alpha}\f = \frac1{\Gamma(\alpha)}\int_0^\infty t^{\alpha-1}e^{-\frac{t}{n}}e^{-t\bm{A}}\f dt,\qquad\;\,\f\in L^2(\R^d;\C^m),
\end{equation}
for every $n\in\N$, where the integral is understood in the improper sense, see, e.g., \cite[Proposition 3.3.5]{haase2006}.

From Propositions \ref{proposition:measurability} and \ref{lemma:rappr_misurabile}, it follows that
\begin{align}
((n^{-1} + \bm{A})^{-\alpha}\f)_i(x)
= \frac1{\Gamma(\alpha)}\sum_{j=1}^m\int_0^\infty t^{\alpha-1}e^{-\frac{t}{n}}dt\int_{\R^d}q_{ij}(t,x,y)f_j(y)dy
\label{order}
\end{align}
for every $\bm f\in L^2(\R^d;\C^m)$. Estimate \eqref{stima_exp_nuclei_q} and H\"older inequality show that the function $(t,y)\mapsto t^{\alpha-1}e^{-\frac{t}{n}}q_{ij}(t,x,y)f_j(y)$ belongs to $L^1((0,\infty)\times\R^d)$ for almost every $x\in\R^d$. Hence, in 
\eqref{order} we can interchange the order of integration and conclude that
\begin{align}
((n^{-1} + \bm{A})^{-\alpha}\f)_i (x)
= & \sum_{j=1}^m\int_{\R^d}\bigg (\frac{1}{\Gamma(\alpha)}\int_0^\infty t^{\alpha-1}e^{-\frac{t}{n}}q_{ij}(t,x,y)dt\bigg ){f_j(y)}dy
\label{lene}
\end{align}
for almost every $x\in\R^d$.

\begin{prop}
\label{prop:conv_nablaH-12}
The sequences $({\nabla}_{\a} (n^{-1}+\bm{A})^{-\frac12})_{n\in\N}$ and $(\bm{L}^{\frac{1}{2}}(n^{-1}+\bm{A})^{-\frac12})_{n\in\N}$ are well defined on $L^2(\R^d;\C^m)$ and converge strongly as $n$ tends to $\infty$ to bounded operators  which are contractions on $L^2(\R^d;\C^m)$ and coincide, respectively, with $\nabla_{\bm a}{\bm A}^{-\frac{1}{2}}$ and ${\bm L}^{\frac12}{\bm A}^{-\frac{1}{2}}$ on $D({\bm A}^{-\frac{1}{2}})$. 
\end{prop}

\begin{proof}
Let $\bm f\in L^2(\R^d;\C^m)$. Then, $(n^{-1}+\bm{A})^{-\frac12}\f$ belongs to $D( (n^{-1}+\bm{A})^{\frac12})=D(\bm{A}^{\frac{1}{2}})=D(\mathfrak a)$ from \cite[Proposition 3.1.9]{haase2006}. From \eqref{eq:nabla:le:square:root}, it follows that 
\begin{equation}
\label{eq:stima:nablaL_e:2-2}
\|\nabla_{\a} (n^{-1}+\bm{A})^{-\frac12}\f\|^2_2\le \|\bm{A}^{\frac12}(n^{-1}+\bm{A})^{-\frac12}\f\|^2_2 \le \|\f\|^2_2,
\end{equation}
where the second inequality follows from \cite[Theorem 7.1.7]{haase2006}, since the function $\varphi_n$ ($n\in\N$), defined by $\varphi_n(z) = z^{\frac12}(n^{-1}+z)^{-\frac12}$, is holomorphic in a sector which contains the halfline $(0,\infty)$, $|\varphi_n|\le 1$, 
and $\bm{A}$ has bounded $\textup{H}^\infty$-calculus on $L^2(\R^d;\C^m)$ (see Corollary \ref{coro:Hinfinity}). Notice that $\varphi_n$ converges to the function identically equal to $1$, uniformly on compact sets of the sector. Moreover, $\bm{A}$ is injective, densely defined, with dense range and $\sup_{n\in\N}\|\varphi_n(\bm{A})\|_{{\mathcal B}(L^2(\R^d;\C^m))}\leq 1$, Hence, from  \cite[Proposition 5.1.4]{haase2006} we conclude that $\bm{A}^{\frac12}(n^{-1}+\bm{A})^{-\frac12}\f$ converges in $L^2(\R^d;\C^m)$ to $\f$ as $n$ tends to $\infty$. In particular, by \eqref{eq:stima:nablaL_e:2-2}  $(\nabla_{\a}(n^{-1}+\bm{A})^{-\frac12}\bm f)_{n\in\N}$ is a Cauchy sequence in $L^2(\R^d;\C^m)$ and this implies that, for every $\f\in L^2(\R^d;\C^m)$, $\nabla_{\a}(n^{-1}+\bm{A})^{-\frac12}\f$ converges in $L^2(\R^d;\C^m)$ as $n$ tends to $\infty$. In this way, we can define a linear operator on $L^2(\R^d;\C^m)$, which by \eqref{eq:stima:nablaL_e:2-2} is a contraction.

Let us prove that, for every $\f\in D(\bm A^{-\frac{1}{2}})$, 
$\nabla_{\a}(n^{-1}+\bm{A})^{-\frac12}\bm f$ converges to
$\nabla_{\a}\bm{A}^{-\frac12}\bm f$ in $L^2(\R^d;\C^m)$. For this purpose, we fix $\f\in D(\bm A^{-\frac{1}{2}})$. Then, there exists $\g\in D(\bm A^\frac12)$ such that $\f={\bm A}^{\frac{1}{2}}\g$. Applying the same argument as above, it can be shown that $(n^{-1}+\bm{A})^{-\frac{1}{2}}\f=(n^{-1}+\bm{A})^{-\frac{1}{2}}{\bm A}^{\frac{1}{2}}\g$ converges to $\g$ in $L^2(\R^d;\C^m)$ as $n$ tends to $\infty$.
Since $\nabla_{\bm a}$ is a closed operator, we deduce that
$\nabla_{\bm a}(n^{-1}+\bm{A})^{-\frac{1}{2}}\f$ converges in $L^2(\R^d;\C^m)$ to $\nabla_{\bm a}\g=\nabla_{\bm a}\bm{A}^{-\frac{1}{2}}\f$.

To complete the proof, we consider the second family of operators in the statement of the proposition, which are well defined and bounded on $L^2(\R^d;\C^m)$ since, by Remark \ref{rem:Lp:square:root:scalar}(iii), $D(\bm{A})\subset D(\bm{L}^{\frac{1}{2}})$  and
\begin{align}
\|\bm{L}^{\frac12}((m^{-1}+\bm{A})^{-\frac12}\bm f-(n^{-1}+\bm{A})^{-\frac12}\bm f)\|_2^2
\le  \|\bm{A}^\frac12((m^{-1}+\bm{A})^{-\frac12}\bm f-(n^{-1}+\bm{A})^{-\frac12}\bm f)\|_2^2,
\label{stima-1}
\end{align}
\begin{equation}\label{eq:bound:2-2:con:eps}
\|\bm{L}^{\frac12}(n^{-1}+\bm{A})^{-\frac12}\f\|^2_2 \le \|\bm{A}^{\frac12}(n^{-1}+\bm{A}
)^{-\frac12}\f\|^2_2 \le \|\f\|^2_2
\end{equation}
for every $m,n\in\N$ and every $\f\in L^2(\R^d;\C^m)$. Estimate \eqref{stima-1} combined with the first part of the proof shows that $(\bm{L}^{\frac12}(n^{-1}+\bm{A})^{-\frac12}\bm f)_{n\in\N}$ is a Cauchy sequence in $L^2(\R^d;\C^m)$. Hence, it converges as $n$ tends to infinity and the norm of the limit does not exceed the norm of $\f$ due to \eqref{eq:bound:2-2:con:eps}. 

The proof that, for every $\bm f\in D(\bm A^{-\frac12})$, the $L^2(\R^d;\C^m)$-limit of $\bm L^\frac12(n^{-1}+\bm A)^{-\frac12}\bm f$ as $n$ tends to infinity coincides with $\bm L^\frac12 \bm A^{-\frac12}\bm f$ can be obtained as above. 
\end{proof}

\begin{defi}
In view of Proposition $\ref{prop:conv_nablaH-12}$, we denote by 
$\nabla_{\bm a}{\bm A}^{-\frac{1}{2}}$ and ${\bm L}^{\frac{1}{2}}{\bm A}^{-\frac{1}{2}}$ the extensions of these operators to $L^2(\R^d;\C^m)$.
\end{defi}

\begin{rem}
In \cite[Lemma 2.21]{dziubanski}, the author proves the boundedness of the operator $(-\Delta)^{\frac12}H^{\frac12}$ on $L^2(\mathbb{R}^d)$ under the assumptions that $d\ge 3$, using a different argument. Clearly, Proposition~\ref{prop:conv_nablaH-12} also applies to the setting considered in \cite{dziubanski}, with no restriction on the dimension $d$.
\end{rem}

\begin{cor}
For every $\f\in L^2(\R^d;\C^m)$, the following factorization
\begin{equation}\label{eq:factorization}
\nabla_{\a}{\bm A}^{-\frac12}\f = (\nabla_{\a}\bm{L}^{-\frac12}){\bm L}^\frac12{\bm{A}}^{-\frac12}\f
\end{equation}
holds almost everywhere on $\R^d$.
\end{cor}

\begin{proof}
For every $\bm f\in L^2(\R^d;\C^m)$, we can write
\begin{align*}
\nabla_{\a}{\bm A}^{-\frac12}\bm f
= \lim_{n\to\infty}\nabla_{\a} (n^{-1}+\bm{A})^{-\frac12}\bm f 
= \lim_{n\to\infty}
\nabla_{\a} \bm{L}^{-\frac12} \bm{L}^\frac12(n^{-1}+\bm{A})^{-\frac12}\bm f,
\end{align*}
where the limit is meant in $L^2(\R^d;\C^m)$. Since $\nabla_{\a} \bm{L}^{-\frac12}$ is a  contraction on $L^2(\R^d;\C^m)$ (see  Remark \ref{rem:Lp:square:root:scalar} (ii)), we obtain easily \eqref{eq:factorization} from Proposition \ref{prop:conv_nablaH-12}. 
\end{proof}

The diamagnetic inequality \eqref{eq:diamagnetic:inequality}, formula \eqref{eq:rep:e+H:-1/2} and Lemma \ref{lemma:rappr_misurabile} yield easily the following domination result.
\begin{lemma}
For every $\bm f\in L^2(\R^d;\C^m)$, every $n\in\N$, every $\alpha\in (0,1)$ and almost every $x\in\R^d$, we have
\begin{align}
\label{cons:diam}
\|((n^{-1}+\bm{A})^{-\alpha}\bm f)(x)\| 
\leq ((n^{-1}+H)^{-\alpha}\|\bm f\|)(x).
\end{align}
\end{lemma}
 
In the last part of this section, we use the following representation formulae for the square root
\begin{equation}
\label{eq:repres:H+eps:12}
(n^{-1}+\bm{A})^{\frac12}\f   = \frac{1}{2\sqrt\pi}\int_0^\infty t^{-\frac32}(1-e^{-t(n^{-1}+\bm{A})})\f dt,\qquad\;\,\f\in D(\bm{A}),\;\, n\in\N,
\end{equation}
and 
\begin{equation}
\label{eq:repres:-Delta:12}
\bm{L}^{\frac12}\f   = \frac{1}{2\sqrt\pi}\int_0^\infty t^{-\frac32}(1-e^{-t \bm{L}})\f dt,\qquad\;\,\f\in D(\bm{L}),
\end{equation}
where the integrals are understood in the improper sense. These integral representations for the square root is due to V. Balakrishnan, see for instance \cite[Chapter IX, Section 11]{yosida}. It is possible to obtain the same representations using the Bochner-Phillips functional calculus, cf. \cite{schilling}, since the function $\lambda\mapsto\lambda^{\frac12}$ has a representation of the form
\begin{align*}
    \lambda^{\frac12} = \frac1{2\sqrt{\pi}} \int_0^\infty t^{-\frac{3}{2}}(1-e^{-\lambda t})dt,\qquad\;\,\lambda\in\C,\;\operatorname{Re}\lambda>0.
\end{align*} 

Proposition \ref{prop:smgr_L1}(ii), applied with $V=0$, shows that the semigroup $(e^{-t\bm{L}})_{t\ge 0}$ extends to a strongly continuous semigroup of contractions $(e^{-t\bm{L}_p})_{t\ge 0}$
in $L^p(\R^d;\C^m)$ for every $p\in [1,\infty)$. In the following lemma we prove a partial characterization of $D(\bm{L}_1)$, which is crucial for the proof of Proposition \ref{prop:bdd:1-1}. 

\begin{lemma}
\label{lemma:sole}
Let $\a\in L^q_{\rm loc}(\R^d;\R^d)$ for some $q\in (2,\infty)$ be such that $\operatorname{div}\a\in L^s_{\rm loc}(\R^d)$ for some $s\in (1,\infty)$. Then, the set $\{\bm u\in L^1(\R^d;\C^m)\cap L^\infty(\R^d;\C^m)\cap \textup{H}^1_\a(\R^d;\C^m): \bm{\Delta}_\a \uu\in L^1(\R^d;\C^m)\}$ is contained in $D(\bm{L}_1)$.
\end{lemma}

\begin{proof}
Let $\uu$ and $\a$ be as in the statement of the lemma. We split the proof into two steps. In the first step, we prove that 
\begin{align}
\label{diamagnetica_epsilon}
\Delta |u_j|\geq {\rm Re}(|u_j|^{-1}\overline{u_j}\Delta_{\a}u_j)
\end{align}
for every $j\in\{1,\ldots,m\}$. In the second step, we complete the proof.

{\em Step 1}. To begin with, we observe that $\nabla \bm u=\nabla_\a\bm u-i\a \bm u\in L^2_{\rm loc}(\R^d;\C^{dm})$. 

Let $(\varrho_n)_{n\in\N}$ be a standard sequence of mollifiers. Clearly, $\bm u\star\varrho_n$ converges to $\bm u$ in $L^r(\R^d;\C^m)$ for every $r\in[1,\infty)$. Fix $j\in\{1,\ldots,m\}$. Since
$\Delta_\a u_j=\Delta u_j-2i\a\cdot \nabla u_j-i(\operatorname{div}\a)u_j-|\a|^2u_j$ in the distributional sense and  $(\Delta u_j)\star\varrho_n=\Delta (u_j\star\varrho_n)$, it follows that 
\begin{align*}
(\Delta_\a u_j)\star\varrho_n-\Delta_\a(u_j\star\varrho_n)
= &  -2i[(\a\cdot \nabla u_j)\star\varrho_n-(\a\cdot \nabla(u_j\star\varrho_n))]\\
&-i[((\operatorname{div}\a)u_j)\star\varrho_n-(\operatorname{div} \a)(u_j\star\varrho_n)]\\
&-[(|\a|^2u_j)\star\varrho_n-|\a|^2(u_j\star\varrho_n)]
\end{align*}
for every $n\in\N$.
The assumptions on $\a$ and on $\bm u$ guarantee that the functions
$\a\cdot\nabla u_j$, $(\operatorname{div}\a)u_j$ and $|\a|^2u_j$ belong to $L^1_{\rm loc}(\R^d;\C)$ and that the right-hand side of the previous equality converges to zero in $L^1_{\rm loc}(\R^d;\C)$ as $n$ tends to $\infty$. It thus follows that $\Delta_{\a}(u_j\star\varrho_n)$ converges to $\Delta_\a u_j$ in $L^1_{\rm loc}(\R^d;\C)$ as $n$ tends to $\infty$.

We can now prove \eqref{diamagnetica_epsilon}, adapting the arguments in the proof of \cite[Lemma A]{kato:1972}. For this purpose, we recall that, if 
$v\in C^2(\R^d;\C)$, then the proof of \cite[Lemma 3]{kato:1972} shows that 
\begin{align}
\Delta\varphi_{\varepsilon}(v)\geq {\rm Re}(\varphi_{\varepsilon}(v)^{-1}\overline v\Delta_{\a}v), \qquad     \varepsilon>0,
\label{diamagnetic-kato}
\end{align}
where we set $\varphi_{\varepsilon}(v)=(|v|^2+\varepsilon)^\frac{1}{2}$.
Applying \eqref{diamagnetic-kato} with $v=u_j\star\varrho_n$ we obtain
\begin{align}
\label{diamagnetica_rho_epsilon}
\Delta \varphi_{\varepsilon}(u_j\star\varrho_n)\geq {\rm Re}(|\varphi_{\varepsilon}(u_j\star\varrho_n)|^{-1}(\overline u_j\star\varrho_n)\Delta_{\a}(u_j\star\varrho_n)).
\end{align}
Since $\varphi_{\varepsilon}$ is Lipschitz continuous, $\varphi_{\varepsilon}(u_j\star\varrho_n)$ converges to $\varphi_{\varepsilon}(u_j)$ in $L^1_{\rm loc}(\R^d;\C)$ as $n$ tends to $\infty$. Consequently, $\Delta \varphi_{\varepsilon}(u_j\star\varrho_n)$ converges to $\Delta\varphi_{\varepsilon}(u_j)$ in the sense of distributions. Similarly, the right-hand side of \eqref{diamagnetica_rho_epsilon} converges to 
${\rm Re}(|\varphi_{\varepsilon}u_j|^{-1}\overline{ u_j}\Delta_{\a}u_j)$ in $L^1_{\rm loc}(\R^d;\C)$ as $n$ tends to $\infty$.

Letting $n$ tend to $\infty$ in \eqref{diamagnetica_rho_epsilon} and then $\varepsilon$ to $0^+$, we obtain \eqref{diamagnetica_epsilon}.

{\em Step 2.} We now complete the proof, showing that $\uu$ belongs to $D(\bm{L}_1)$.
Set $\bm f=\bm u-\Delta_\a \bm u$, which is a function in $L^1(\R^d;\C^m)$, and let $(\bm f_n)_{n\in\N}\subset C^{\infty}_c(\R^d;\C^m)$ be a sequence which converges to  $\bm f$ in $L^1(\R^d;\C^m)$. For every $n\in\N$, set $\bm v_n=(1+\bm L_1)^{-1}\bm f_n$. Clearly, $\bm v_n$ converges to $\bm v=(1+\bm L_1)^{-1}\bm f$ in $D(\bm L_1)$. 

We note that $\bm v_n\in L^{\infty}(\R^d;\C^m)$, since by the classical diamagnetic inequality for the magnetic Laplacian (which follows also from \eqref{weak-diamagnetic}, taking $V\equiv 0$) we infer that $\|(1+\bm{L}_1)^{-1}\f_n\|\le (1-\Delta)^{-1}\|\f_n\|$ almost everywhere in $\R^d$, and the Gauss-Weierstrass semigroup maps $L^{\infty}(\R^d;\C^m)$ into itself. Moreover, from the consistency of the semigroups $\bm \Delta_\a\bm v_n=-\bm L_1\bm v_n\in L^1(\R^d;\C^m)$ and $\bm v_n\in \textup{H}^1_\a(\R^d;\C^m)$. 

Applying Step 1 to the function $\bm w_n=\bm u-\bm v_n$, which solves $\bm w-\bm \Delta_\a\bm w=\bm f-\bm f_n$, we obtain
\begin{align*}
\Delta |w_{n,j}|\geq {\rm Re}(|w_{n,j}|^{-1}\overline{w_{n,j}}\Delta_\a w_{n,j})
= |w_{n,j}|-{\rm Re}(|w_{n,j}|^{-1}\overline{w_{n,j}}(f_j-f_{n,j}))
\end{align*}
for every $j\in\{1,\ldots,m\}$.

Fix $j$. Since $\bm{w}_n$ converges to $\bm u-\bm{v}$ in $L^1(\R^d;\C^m)$, as $n$ tends to $\infty$, it follows that  $\Delta|w_{n,j}|$ converges to $\Delta|u_j-v_j|$ in the sense of distributions, while ${\rm Re}(|w_{n,j}|^{-1}\overline{w_{n,j}}(f_j-f_{n,j}))$ converges to zero in $L^1(\R^d;\C)$. We have so proved that $\Delta|u_j-v_j|\geq |u_j-v_j|$ in the sense of distributions. By \cite[Lemma 2]{kato:1986}, this implies $|u_j-v_j|=0$, i.e., $u_j=v_j$ almost everywhere in $\R^d$. We have so proved that $\uu=\bm v$, i.e., that $\bm u$ belongs to $D(\bm{L}_1)$ as claimed.
\end{proof}

In view of Lemma \ref{lemma:sole}, we refine Hypothesis \ref{hyp-operatore}(i) as follows.
\begin{hyp}
\label{hyp:a_regolare}
$\a\in L^q_{\rm loc}(\R^d;\R^d)$ for some $q\in (2,\infty)$ and $\operatorname{div}\a\in L^s_{\rm loc}(\R^d)$ for some $s\in (1,\infty)$.
\end{hyp}
In the rest of the section, we assume Hypotheses \ref{hyp-operatore}, \ref{hyp:avl_potenziale} and \ref{hyp:a_regolare}.

\begin{rem}
\begin{enumerate}[\rm (i)]
\item If $d=1$, then we can consider the weaker assumption $\bm a\in W^{1,1}_{\rm loc}(\R)$.
\item It is easy to see that, Hypothesis \ref{hyp:a_regolare} is verified if $\a\in W^{1,r}(\R^d;\R^d)$, with $r>1$, if $d=2$, and with $r>\frac{2d}{d+2}$, if $d\geq 3$. 
\item Inequality \eqref{diamagnetica_epsilon} has been proved in a more general form in \cite{kato:1972} for functions $v\in L^1_{\rm loc}(\R^d;\C)$ such that $\Delta_\a v\in L^1_{\rm loc}(\R^d;\C)$ under the assumption $\a\in C^1(\R^d;\R^d)$. Since we are interested in \eqref{diamagnetica_epsilon} for smoother functions $v$, then we are able to weaken the conditions on $\a$.
\end{enumerate}
\end{rem}

\begin{prop}
\label{prop:bdd:1-1}
The operator $\bm L^\frac12{\bm{A}}^{-\frac12}$ has a unique extension to a bounded operator on $L^1(\R^d;\C^m)$ and
\begin{equation}\label{eq:stima:L1-L1}
\|\bm L^\frac12{\bm{A}}^{-\frac12}\f\|_1 \le (1+\kappa)\|\f\|_1, \qquad\;\, \bm f\in L^1(\R^d;\C^m),
\end{equation}
where $\kappa$ is the constant in Hypothesis $\ref{hyp:avl_potenziale}$.
\end{prop}

\begin{proof}
We will prove that ${\bm L}^{\frac{1}{2}}\bm{A}^{-\frac{1}{2}}\f\in L^1(\R^d;\C^m)$ and \eqref{eq:stima:L1-L1} holds for every $\f$ which belongs to $(1+\bm{A}_1)^{-1}(C^\infty_c(\R^d;\C^m))$ (see Proposition \ref{prop:smgr_L1}). 
The density of this space in $L^1(\R^d;\C^m)$ (see again Proposition \ref{prop:smgr_L1}) will allow us to conclude the proof.

Fix $\bm f\in (1+\bm{A}_1)^{-1}(C^\infty_c(\R^d;\C^m))$, $\tau >0$ and $n\in\N$. Notice that $(n^{-1}+\bm{A})^{-\frac12}\bm f\in D(\bm{A})$ since $(1+\bm{A}_1)^{-1}\bm g=(1+\bm{A})^{-1}\bm g$ for every $\bm g\in C^\infty_c(\R^d;\C^m)$. 
Moreover, the bounded analytic semigroup $(e^{-t \bm{L}})_{t\geq0}$ commutes with $\bm{L}$, then it does with $\bm{L}^{\frac12}$ by \cite[Proposition 3.1.1]{haase2006}. Hence, 
from \eqref{eq:repres:H+eps:12} and \eqref{eq:repres:-Delta:12}, we infer that
\begin{align}
&e^{-\tau \bm{L}}\bm{L}^{\frac12}(n^{-1}+\bm{A})^{-\frac12}\f\notag\\  
= & \bm{L}^{\frac12}e^{-\tau \bm{L}}(n^{-1}+\bm{A})^{-\frac12}\f\notag\\
= & \frac{1}{2\sqrt\pi}\int_{0}^\infty t^{-\frac{3}{2}}(1-e^{-t \bm{L}})e^{-\tau \bm{L}}(n^{-1}+\bm{A})^{-\frac12}\f dt\notag\\
= & \frac{e^{-\tau \bm{L}}}{2\sqrt\pi}\lim_{\delta\to 0}\bigg[\int_{\delta}^\infty t^{-\frac{3}{2}}(1-e^{-t(n^{-1}+\bm{A})})(n^{-1}+\bm{A})^{-\frac12}\f dt\notag\\
&\qquad\qquad\quad - \int_{\delta}^\infty t^{-\frac{3}{2}}(e^{-t \bm{L}}-e^{-t(n^{-1}+\bm{A})})(n^{-1}+\bm{A})^{-\frac12}\f dt\bigg]\notag\\
=& e^{-\tau \bm{L}}\f - \lim_{\delta\to 0}\frac{e^{-\tau \bm{L}}}{2\sqrt\pi}\int_{\delta}^\infty t^{-\frac{3}{2}}(e^{-t \bm{L}}-e^{-t(n^{-1}+\bm{A})})(n^{-1}+\bm{A})^{-\frac12}\f dt.
\label{eq:split:semigruppi}   
\end{align}

To estimate the integral in last side of 
\eqref{eq:split:semigruppi}, we begin by observing that $\f= (1+\bm{A}_1)^{-1}{\bm g}$ for some ${\bm g}\in C^{\infty}_c(\R^d;\C^m)$. From the consistency of the semigroups $(e^{-t\bm L_p})_{t\geq0}$ and $(e^{-t\bm{A}_p})_{t\geq0}$, for $p\in [1,\infty)$, the diamagnetic inequality \eqref{weak-diamagnetic} and the integral representation \eqref{eq:rep:e+H:-1/2} with $\alpha=\frac12$ and $\alpha=1$, it follows that both $e^{-t\bm L}(n^{-1}+\bm A)^{-\frac12}\bm f$ and $e^{-t\bm A}(n^{-1}+\bm A)^{-\frac12}\f$ belong to $L^1(\R^d;\C^m)\cap L^{\infty}(\R^d;\C^m)$ for every $t\in[0,\infty)$ (see Step $2$ in the proof of Lemma \ref{lemma:sole} for further details).
Further, from the integral representation \eqref{eq:rep:e+H:-1/2} of $(n^{-1}+\bm{A})^{-\frac12}$, it follows that
\begin{align*}
& \bm{A}^k_1e^{-t(n^{-1}+\bm{A})}(n^{-1}+\bm{A})^{-\frac12}\bm f= e^{-t(n^{-1}+\bm{A}_1)}(n^{-1}+\bm{A}_1)^{-\frac12}\bm{A}_1^k\bm f, \qquad\;\, t\geq 0,\;\, k=0,1,
\end{align*}
so that the function $t\mapsto e^{-t(n^{-1}+\bm{A})}(n^{-1}+\bm{A})^{-\frac12}\bm f$ belongs to $C([0,\infty);D(\bm{A}_1))$. In particular, \eqref{eq:maximal_ineq:L1_nuovo} with $V+n^{-1}$ instead of $V$ shows that
\begin{align*}
&\|(n^{-1}+V)e^{-t(n^{-1}+\bm{A})}(n^{-1}+\bm{A})^{-\frac12}\bm f\|_1\\
\leq & \kappa\|(n^{-1}+\bm{A}_1)e^{-t(n^{-1}+\bm{A})}(n^{-1}+\bm{A})^{-\frac12}\bm f\|_1
\end{align*}
for every $t\in [0,\infty)$, so that both the functions 
$t\mapsto Ve^{-t(n^{-1}+\bm{A})}(n^{-1}+\bm{A})^{-\frac12}\bm f$ and
$t\mapsto (n^{-1}+V)e^{-t(n^{-1}+\bm{A})}(n^{-1}+\bm{A})^{-\frac12}\bm f$ belong to 
$C([0,\infty);L^1(\R^d;\C^m))$. 

Since ${\bm A}\bm g=-\bm\Delta_{\bm a}\bm g+V\bm g$ in the sense of distributions for every $\bm{g}\in D(\bm{A})$, it follows that 
$\bm{\Delta}_{\bm a} e^{-t(n^{-1}+\bm{A})}(n^{-1}+\bm{A})^{-\frac12}\bm f$ belongs to $C([0,\infty);L^1(\R^d;\C^m))$ as well.

We claim that 
$\uu(t)=(e^{-t \bm{L}}-e^{-t(n^{-1}+\bm{A})})(n^{-1}+\bm{A})^{-\frac12}\bm f$ belongs to $D(\bm{L}_1)$ for every $t\in [0,\infty)$. For this purpose, it suffices to apply Lemma \ref{lemma:sole} since $\uu(t)\in L^1(\R^d;\C^m)\cap L^{\infty}(\R^d;\C^m)$ and $\bm{\Delta}_{\a}\uu(t)\in L^1(\R^d;\C^m)$ for every $t\ge 0$, by the previous part of the proof. Moreover,
since the semigroups in $L^2(\R^d;\C^m)$ are defined through forms, whose domain is contained in $\textup{H}^1_{\a}(\R^d;\C^m)$, $\uu(t)$ belongs to this space as well. 

From the above remarks, we infer  that
the function $t\mapsto (e^{-t \bm{L}}-e^{-t(n^{-1}+\bm{A})})(n^{-1}+\bm{A})^{-\frac12}\bm f$  is a classical solution to the Cauchy problem
\begin{align*}
\left\{
\begin{array}{ll}
\bm u'(t)=-\bm{L}_1 \bm u(t)+(n^{-1}+V)e^{-t(n^{-1}+\bm{A})}(n^{-1}+\bm{A})^{-\frac12}\bm f, & t>0,\\[1mm]
{\bm u}(0)= {\bm 0},
\end{array}
\right.
\end{align*}
in $L^1(\R^d;\C^m)$.
We can thus apply Duhamel formula (\cite[Section 4.2, Corollary 2.2]{pazy}), to obtain that 
\begin{align*}
&(e^{-t \bm{L}}-e^{-t(n^{-1}+\bm{A})})(n^{-1}+\bm{A})^{-\frac12}\bm f\\
=&\int_0^te^{-(t-s)\bm{L_1}}((n^{-1}+V)e^{-s(n^{-1}+\bm{A})}(n^{-1}+\bm{A})^{-\frac12}\bm f)ds
\end{align*}
for every $t\in (0,\infty)$.
By replacing in \eqref{eq:split:semigruppi}, recalling the consistency of the semigroups $(e^{-t\bm L_p})_{t\geq0}$, $p\in[1,\infty)$, taking the modulus and integrating with respect to $x$ on $\R^d$, we infer that
\begin{align*}
& \bigg\|e^{-\tau \bm{L}}\int_\delta^\infty t^{-\frac{3}{2}}(e^{-t\bm{L}}-e^{-t(n^{-1}+\bm{A})})(n^{-1}+\bm{A})^\frac12\bm f dt\bigg\|_1  \\
\leq & \int_\delta^\infty t^{-\frac{3}{2}}dt\int_0^t\|e^{-(\tau+t-s)\bm{L_1}}(n^{-1}+V)e^{-s(n^{-1}+\bm{A})}(n^{-1}+\bm{A})^{-\frac12}\bm f\|_1ds  \\
\leq & \int_\delta^\infty t^{-\frac{3}{2}}dt\int_0^t \|(n^{-1}+V)e^{-s(n^{-1}+\bm{A})}(n^{-1}+\bm{A})^{-\frac12}\bm f\|_1ds,
\end{align*}
since $e^{-t\bm{L_1}}$ is a contraction on $L^1(\R^d;\C^m)$ for every $t\geq0$. Further, from Hypothesis \ref{hyp:avl_potenziale}, estimates \eqref{eq:diamagnetic:inequality} and \eqref{cons:diam} and formula \eqref{eq:rep:e+H:-1/2}, it follows that
\begin{align*}
& \int_{\delta}^\infty t^{-\frac{3}{2}}dt\int_0^t\|(n^{-1}+V)e^{-s(n^{-1}+\bm{A})}(n^{-1}+\bm{A})^{-\frac12}\bm f\|_1dx\\
= & \int_0^\infty ds\int_{\R^d}\|(n^{-1}+V(x))(e^{-s(n^{-1}+\bm{A})}(n^{-1}+\bm{A})^{-\frac12}\bm f)(x)\|dx\int_{\max\{s,\delta\}}^\infty t^{-\frac32}dt  \\
\leq & 2\kappa\int_{\R^d}(n^{-1}+\lambda_V(x))dx\int_0^\infty s^{-\frac{1}{2}}(e^{-s(n^{-1}+H)}\|(n^{-1}+\bm{A})^{-\frac12}\bm f\|)(x)ds \\
\leq & 2\kappa\int_{\R^d}(n^{-1}+\lambda_V(x))dx\int_0^\infty s^{-\frac{1}{2}}(e^{-s(n^{-1}+H)}(n^{-1}+H)^{-\frac12}\|\bm f\|)(x) ds \\
= & 2\kappa\sqrt \pi\|(n^{-1}+\lambda_V)(n^{-1}+H)^{-1}\|\bm f\|\|_1\\
\le & 2\kappa\sqrt \pi\|\bm f\|_1,
\end{align*}
where the last inequality follows from  \cite[Lemma 6]{kato:1986}. From the above estimates we conclude that
\begin{align}
\bigg\| \frac{1}{2\sqrt \pi}e^{-\tau \bm{L}}\int_\delta^\infty t^{-\frac{3}{2}}(e^{-t\bm{L}}-e^{-t(n^{-1}+\bm{A})})(n^{-1}+\bm{A})^\frac12\bm fdt\bigg \|_1
\leq \kappa\|\bm f\|_1
\label{nome}
\end{align}
for every $\delta>0$. By replacing \eqref{nome} in \eqref{eq:split:semigruppi}, we infer that
\begin{align*}
\|e^{-\tau \bm{L}}\bm{L}^{\frac12}(n^{-1}+\bm{A})^{-\frac12}\bm f\|_1
\leq (1+\kappa)\|\bm f\|_1.
\end{align*}

Since $e^{-\tau \bm{L}}\bm{L}^{\frac12}(n^{-1}+\bm{A})^{-\frac12}\f$ converges to $\bm{L}^{\frac12}(n^{-1}+\bm{A})^{-\frac12}\f$ in $L^2(\R^d;\C^m)$ as $\tau$ tends to $0^+$, there exists a decreasing null sequence $(\tau_n)_{n\in\N}$ such that $e^{-\tau_n\bm{L}}\bm{L}^{\frac12}(n^{-1}+\bm{A})^{-\frac12}\f$ converges to $\bm{L}^{\frac12}(n^{-1}+\bm{A})^{-\frac12}\f$ almost everywhere in $\R^d$. By applying Fatou's Lemma, we conclude that
\begin{equation}
\|\bm{L}^{\frac12}(n^{-1}+\bm{A})^{-\frac12}\bm f\|_1\leq (1+\kappa)\|\bm f\|_1.
\label{Fatou}\end{equation}

Since the sequence $\bm{L}^{\frac12}(n^{-1}+\bm{A})^{-\frac12}\bm f)$, converges to the function $\bm L^\frac12{\bm{A}}^{-\frac12}\f$ in $L^2(\R^d;\C^m)$ as $n$ tends to $\infty$, up to a subsequence, the convergence is almost everywhere in $\R^d$. Applying Fatou's lemma to estimate
\eqref{Fatou}, we complete the proof. 
\end{proof}

\begin{theo} 
The operators $\bm L^\frac12\bm A^{-\frac12}$ and $\nabla_\a{\bm{A}}^{-\frac12}\f$ admit bounded extensions to $L^p(\R^d;\C^m)$ for all $p\in(1,2]$. In addition,
\begin{equation*}
\|\bm L^\frac12\bm A^{-\frac12}\f\|_p\leq (1+\kappa)^{\frac{2-p}{p}}\|\f\|_p, \qquad 
\|\nabla_\a{\bm{A}}^{-\frac12}\f\|_p \le c'(1+\kappa)^{\frac{2-p}{p}}\|\f\|_p
\end{equation*}
for every $\f\in L^p(\R^d;\C^m)$, where $\kappa$ is the constant in Hypothesis $\ref{hyp:avl_potenziale}$ and $c'$ is a positive constant.
\end{theo}

\begin{proof}
Since, from Propositions \ref{prop:conv_nablaH-12} and \ref{prop:bdd:1-1}, $\bm{L}^{\frac{1}{2}}\bm{A}^{-\frac{1}{2}}\in {\mathcal B}(L^1(\R^d;\C^d))\cap {\mathcal B}(L^2(\R^d;\C^m))$, the Riesz-Thorin interpolation Theorem (see, e.g., \cite[Theorem 4.5.2]{grafakos-classic}) implies that $\bm L^{\frac12}{\bm{A}}^{-\frac12}$ is bounded on $L^p(\R^d;\C^m)$ for every $p\in (1,2)$ and 
\begin{equation}\label{eq:stima:Lp-Lp}
\|\bm L^{\frac12}{\bm{A}}^{-\frac12}\|_{{\mathcal B}(L^p(\R^d;\C^m))} \le (1+\kappa)^{\frac{2-p}{p}}.
\end{equation}
As already observed in Remark \ref{rem:Lp:square:root:scalar}(ii), the operator $\nabla_\a \bm{L}^{-\frac12}$ acts componentwise. Moreover, $\nabla_\a L^{-\frac12}$ is a bounded operator on $L^p(\R^d;\C)$ for every $p\in(1,2]$ (see,  \cite{sikora, DOY06, BenAli1}). Hence, 
taking \cite[Proposition 3.1]{IM96} into account, we can estimate 
\begin{align}
\|\nabla_\a \bm{L}^{-\frac12}\uu\|_p =& \bigg\|\bigg (\sum_{k=1}^m|\nabla_\a L^{-\frac12}u_j|^2\bigg )^{\frac12}\bigg \|_p
\le \|\nabla_\a L^{-\frac12}\|_{{\mathcal B}(L^p(\R^d;\C))}\|\uu\|_p\eqqcolon c'\|\uu\|_p
\label{eq:stima:p-p:riesz:lap}
\end{align}
for every $\uu\in L^p(\R^d;\C^m)$.

Let $p$ be fixed in $(1,2]$ and let $\f$ be in $C^\infty_c(\R^d;\C^m)$. From the factorization \eqref{eq:factorization}, estimates \eqref{eq:stima:Lp-Lp} and \eqref{eq:stima:p-p:riesz:lap}, we conclude that
\begin{align*}
\|\nabla_\a{\bm{A}}^{-\frac12}\f\|_p= \|(\nabla_\a \bm{L}^{-\frac12})\bm L^{\frac12}{\bm{A}}^{-\frac12}\f\|_p
\le c'\|\bm L^{\frac12}{\bm{A}}^{-\frac12}\f\|_p\le c'(1+\kappa)^{\frac{2-p}{p}}\|\f\|_p.
\end{align*}

The density of $C^\infty_c(\R^d;\C^m)$ in $L^p(\R^d;\C^m)$ allows us to complete the proof.
\end{proof}

\begin{rem}
\begin{enumerate}[\rm (i)]
\item 
Suppose $\kappa$ is independent of $m$. Then the Riesz-transform estimate for $A$ is independent of the number of components $m$: i.e., the operator norm bound is uniform in $m$ for every potential in the class determined by Hypothesis \ref{hyp:avl_potenziale} with the fixed constant $\kappa$.
\item 
This estimate is also independent of the potential, in the sense that it holds for every matrix potential $V$ satisfying Hypothesis \ref{hyp:avl_potenziale} with the same constant $\kappa$.
\item 
Notice that all the results obtained so far for ${\bm A}$ remain valid when $\bm a = \bm 0$. In particular, for $p\in(1,2]$, the estimate
\begin{equation*}
\|\nabla {(-\bm \Delta)}^{-\frac12}\bm f\|_p \le \frac{c}{p-1}\|\bm \f\|_p
\end{equation*}
is sharp for every $\bm f \in L^p(\R^d;\C^m)$, where  $c$ is a positive constant independent of the dimension, see \cite[Theorem 1.1]{IM96} or \cite[Theorem 3]{BW95}. It follows that the Riesz transform of the purely electric Schr\"odinger operator satisfies the estimate
\begin{equation*}
\|\nabla(-\bm \Delta+V)^{-\frac12} \bm f\|_p \le \frac{c(1+\kappa)^{\frac{2-p}{p}}}{p-1}\|\bm f\|_p
\end{equation*}
for every $\bm f \in L^p(\R^d;\C^m)$ and $p\in(1,2]$. Here, $\kappa$ denotes the constant appearing in Hypothesis \ref{hyp:avl_potenziale}. Hence, if $\kappa$ does not depend on $d$, then as in \cite{dziubanski} we recover a dimension-free estimate for the Riesz transform, which is also potential-free in the sense that it holds for every matrix potential $V$ which satisfies Hypothesis \ref{hyp:avl_potenziale} with the same constant $\kappa$.
\end{enumerate}
\end{rem}

\section{\texorpdfstring{$V$}{V} zero-order Riesz transform}
\label{sec:0:RT}
In this section, assuming Hypotheses \ref{hyp-operatore} and \ref{hyp:avl_potenziale},we define the Riesz transform $V^{\alpha}\bm{A}^{-\alpha}$ and prove that it is bounded on $L^p(\mathbb{R}^d;\mathbb{C}^m)$ for every $p\in(1,2]$ and every $\alpha\in[0,1/p]$. 
The form method shows that the operator $V^{\frac{1}{2}}\bm{A}^{-\frac{1}{2}}$ is well defined on $D(\bm{A}^{-\frac{1}{2}})$. 

For other values of $\alpha$ and $p$, the definition of the operator $V^{\alpha}\bm{A}^{-\alpha}$ is not straightforward, due to the fact that the entries of the potential $V$ are only locally integrable functions. We therefore define the operator $V^{\alpha}\bm{A}^{-\alpha}$ by means of an approximation procedure.

For every $n\in\N$, we consider the 
sequence $(V_n)_{n\in\N}$ of bounded potential introduced at the beginning of Section \ref{sec:max_ineq_Lp}, which is defined by
\begin{align*}
V_n(x)\coloneqq{\mathds 1}_{F_n}(x) V(x)+(1-{\mathds 1}_{F_n}(x))(\lambda_V(x)\wedge n)\operatorname{Id}_{\R^m},\qquad\;\,x\in\R^d,    
\end{align*}
where $F_n=\{x\in\R^d: n>\kappa\lambda_V(x)\}$. Moreover, we set $\bm{R}^{0,n}\coloneqq\pmb{\operatorname{Id}}$ and
\begin{eqnarray*} 
\bm{R}^{\alpha,n}\bm f= V_n^{\alpha}
(n^{-1}+\bm{A})^{-\alpha}\bm f
\end{eqnarray*}
for every simple function $\bm f$ with bounded support, if $\alpha>0$. 

\begin{rem}
\label{rem:V:eps:bip}
For every $n\in\N$ and almost every $x\in \R^d$, the matrix $V_n(x)$, seen as an operator from $\C^m$ into itself, admits bounded imaginary powers $(V_n(x))^{is}$ and $\|(V_n(x))^{is}\|_{{\mathcal B}(\C^m)}\le1$ for every $s\in\R$. This is a consequence of \cite[Corollary 7.1.6(a)]{haase2006}, which shows that the imaginary powers of positive self-adjoint operators in a Hilbert space have unitary norm. Further, from the Komatzu representation formula, it follows that $V_n^{is}(x)=(V_n(x))^{is}$ for almost every $x\in\R^d$ and every $s\in\R$, where $V_n^{is}$ is the imaginary power of $V_n$, seen as a bounded linear operator from $L^p(\R^d;\C^m)$ ($p\in[1,\infty)$) into itself. 
\end{rem}

\begin{prop}
\label{prop:bdd:R:a,e:1/p}
For every $n\in\N$ and every $p\in [1,2]$ the operator $\bm{R}^{\frac1p,n}$ is bounded on $L^p(\R^d;\C^m)$ by a positive constant, depending on $p$ but  independent of $n$.
\end{prop}

\begin{proof}
For every $z\in\Pi_{1/2}^1=\left\{z\in\C:\frac12<{\rm Re}(z)<1\right\}$ and for every $n\in\N$, the function $\lambda\mapsto (n^{-1}+\lambda)^{-z}$ is bounded on $[0,\infty)$ and holomorphic in a sector  containing $[0,\infty)$. Moreover, since the operator $\bm{A}$ admits bounded $\textup{H}^\infty$-calculus on $L^2(\R^d;\C^m)$, the operator $(n^{-1}+\bm{A})^{-z}$ is well-defined and bounded on $L^2(\R^d;\C^m)$. 

Let $\{\bm{T}_z: z\in\overline \Pi_{1/2}^1\}$ be the analytic family defined by
\begin{align}
\label{eq:def:T:z}
\bm{T}_z\bm f= V_n^z
(n^{-1}+\bm{A})^{-z}\bm f, \qquad\;\, \bm f\in L^2(\R^d;\C^m).    
\end{align}

We aim to apply the Stein interpolation Theorem \ref{theo:vv:stein}. For this purpose, we begin by noticing that
the matrix $V_n(x)$ is symmetric and positive definite for every $x\in\R^d$ and it satisfies 
the estimate \eqref{eq:Vn:ge:forma}. It follows that $\|(V_n(x))^{\alpha}\|_{{\mathcal B}(\C^m)}\le \kappa^{\alpha} n^{\alpha}$ for every $\alpha>0$.  Taking also Remark \ref{rem:V:eps:bip} into account, we can estimate
\begin{align*}
|\langle \bm{T}_z\bm f,\bm g\rangle_{L^2(\R^d;\C^m)}| 
&\le \|V_n^{i\operatorname{Im}(z)}\|_{{\mathcal B}(L^2(\R^d;\C^m))}\|V_n^{\operatorname{Re}(z)}\|_{{\mathcal B}(L^2(\R^d;\C^m))}\|(n^{-1}+\bm{A})^{-z}\f\|_2\|\bm g\|_2\\
&\le (\kappa n)^{\operatorname{Re}(z)} 
\|(n^{-1}+\bm{A})^{-i\operatorname{Im}(z)}\|_{ {\mathcal B}(L^2(\R^d;\C^m))}\|(n^{-1}+\bm{A})^{-\operatorname{Re}(z)}\f\|_2\|\bm g\|_2\\
&\le (\kappa n)^{\operatorname{Re}(z)}\frac1{|\Gamma(z)|}\bigg (\int_0^\infty t^{{\rm Re} (z)-1}e^{-\frac{t}{n}}\|e^{-tA}\f\|_2dt\bigg )\|\bm g\|_2\\
&\le \kappa n^{2\operatorname{Re}(z)}\frac{\Gamma(\operatorname{Re}(z))}{|\Gamma(z)|}\|\f\|_2\|\bm g\|_2\\
&\le c\kappa n(\operatorname{cosh}(\pi\operatorname{Im}(z)))^{\frac12}\|\f\|_2\|\bm g\|_2\\
& \le c\kappa n\|\f\|_2\|\bm g\|_2,
\end{align*}
for some positive constant $c$, independent of $\f$ and $\bm{g}$.
This implies that $\{\bm{T}_z: z\in\overline\Pi_{1/2}^1\}$ is an analytic family of operators with admissible growth. 

It remains to  bound $\bm{T}_z$ for ${\rm Re}(z)=\frac12$ and ${\rm Re}(z)=1$ on $L^{2}(\R^d;\C^m)$ and $L^{1}(\R^d;\C^m)$, respectively. For this purpose, we fix a simple function $\f$ with bounded support and notice that from the domination \eqref{eq:diamagnetic:inequality} and Remark \ref{rem:V:eps:bip} it follows that
\begin{align*}
\|(\bm{T}_{1+is}\bm f)(x)\|
& \leq  \frac{\|(V_n(x))^{is}\|_{{\mathcal B}(\C^m)}}{|\Gamma(1+is)|}
\bigg\|V_n(x)\bigg (\int_0^\infty t^{is}e^{-\frac{t}{n}}e^{-t\bm{A}}\bm fdt\bigg )(x)\bigg \| \\
& \leq  \frac{\kappa\sqrt{\sinh(\pi |s|)}}{\sqrt{\pi |s|}}(\lambda_V(x)\wedge n)\bigg (\int_0^\infty e^{-\frac{t}{n}}e^{-tH}\|\bm f\|dt\bigg )(x) \\
& =   \frac{\kappa\sqrt{\sinh(\pi |s|)}}{\sqrt{\pi |s|}}(\lambda_V(x)\wedge n)((n^{-1}+H)^{-1}\|\bm f\|)(x)
\end{align*}
for almost every $x\in\R^d$, 
where we also used the formula $|\Gamma(1+is)|=\sqrt{\pi |s|}(\sinh(\pi |s|))^{-\frac12}$ for every $s\in\R$.
Integrating over $\R^d$ and taking the maximal inequality in $L^1(\R^d)$ for the operator $-(n^{-1}+H)$ into account, we get $\|\bm{T}_{1+is}\bm f\|_1\le b_1(s) \|\bm f\|_1$, where $b_1(s)= \kappa((\pi |s|)^{-1}\sinh(\pi |s|))^{\frac12}$ for every $s\in\R$. 

Analogously, recalling that $|\Gamma\left(\frac12+is\right)|=\sqrt\pi(\cosh(\pi s))^{-\frac12}$ we deduce that
\begin{align*}
\|(\bm{T}_{\frac12+is}\bm f)(x)\|
&\le  b_{\frac12}(s)
(\lambda_V(x)\wedge n)^\frac12((n^{-1}+H)^{-\frac12}\|\bm f\|)(x)
\end{align*}
for almost every $x\in\R^d$, where  $b_{\frac12}(s)= (\pi^{-1}\kappa\cosh{(\pi s)})^{\frac12}$ for every $s\in\R$. 
Hence, taking into account that $(\lambda_V\wedge n)^{\frac{1}{2}}(n^{-1}+H)^{-\frac{1}{2}}$ is a contraction in $L^2(\R^d;\C)$ (a property that follows from the definition of $\mathfrak{h}$ together with $\|\lambda_V^{1/2}f\|_2 \le \|(n^{-1}+H)^{\frac12}f\|_2$ for every $f\in D(\mathfrak{h})$), it can be shown that $\|\bm{T}_{\frac12+is}\bm f\|_2\le b_{\frac12}(s)\|\bm f\|_2$. 
Since the functions $b_1$ and $b_{\frac12}$ satisfy the growth condition in the statement  of Theorem \ref{theo:vv:stein} over the interval $\left(\frac12,1\right)$, by applying that theorem with $a_0=\frac12$, $a_1=1$, $q_0=p_0=2$ and $q_1=p_1=1$, we get that for every $p\in [1,2]$ there exists a positive constant $M$, depending on $p$ but  independent of $n$, such that
\begin{align*}
\|\bm{R}^{\frac{1}{p},n}\bm f\|_p\leq M\|\bm f\|_p
\end{align*}
for every simple function $\bm f$ with bounded support. The conclusion follows by density.
\end{proof}

It is possible to prove the following stronger interpolation result for the operators $\bm{R}^{\alpha,n}$.

\begin{prop}\label{prop:KW:gen}
Let $n\in\N$, $0\le \alpha_0< \alpha_1$ and $p_0,p_1\in(1,\infty)$. Assume that the operators ${\bm R}^{\alpha_j,n}$ are bounded on $L^{p_j}(\R^d;\C^m)$ for $j\in\{0,1\}$. Then, the operator ${\bm R}^{\alpha,n}$ is bounded on $L^{p}(\R^d;\C^m)$ with $p$ and $\alpha$ satisfying
\begin{equation*}
\frac1p = \frac{\theta}{p_1}+\frac{1-\theta}{p_0},\qquad\;\, \alpha= \alpha_1\theta +  \alpha_0(1-\theta),\qquad\theta\in(0,1).
\end{equation*}
In particular, if the norm of ${\bm R}^{\alpha_j,n}$ on $L^{p_j}(\R^d;\C^m)$, $j\in\{0,1\}$, is bounded by a constant which is independent of $n$, then so does for the norm of ${\bm R}^{\alpha,n}$ as bounded linear operator on $L^p(\R^d;\C^m)$.
\end{prop}

\begin{proof}
Throughout the proof, by $c$ we denote a positive constant, which depends at most on $m$, $p_0$ and $p_1$ but does not depend on $n$ and may vary from line to line.
    
We consider the analytic family $\{\bm{T}_z: z\in\overline{\Pi}_{\alpha_0}^{\alpha_1}\}$, where the operator $\bm{T}_z$ has been introduced in \eqref{eq:def:T:z}. Arguing as before, this is an analytic for every simple function with admissible growth and, for every $\alpha\in [\alpha_0,\alpha_1]$, every  $n\in\N$ and every $\f\in L^2(\R^d;\C^m)$, we have $\bm{T}_\alpha\bm f=\bm{R}^{\alpha,n}\bm f$.
To apply again Theorem \ref{theo:vv:stein}, we have to bound $\bm{T}_z$ for ${\rm Re}(z)= \alpha_0$ and ${\rm Re}(z)= \alpha_1$ on $L^{p_0}(\R^d;\C^m)$ and on $L^{p_1}(\R^d;\C^m)$, respectively.

Notice that, writing $z=\alpha_j+is$, for $s\in\R$ and $j\in\{0,1\}$, it follows that
\begin{align*}
\bm{T}_z = V_n^{is}\bm{T}_{\alpha_j}(n^{-1}+\bm{A})^{-is}.
\end{align*}
This implies that $\|\bm{T}_{ \alpha_0+is}\bm f\|_{p_0}\le \|\bm{T}_{ \alpha_0}\|_{{\mathcal B}(L^{p_0}(\R^d;\C^m))}\|(n^{-1}+\bm{A})^{-is}\bm f\|_{p_0}$ for every simple function $\f$. By Corollary \ref{coro:Hinfinity}, the operator $\bm{A}$ admits bounded $\textup{H}^\infty$-calculus in $L^q(\R^d;\C^m)$ for every $q\in(1,\infty)$. Hence, $\|(n^{-1}+\bm{A})^{-is}\|_{{\mathcal B}(L^{p_0}(\R^d;\C^m))}\leq c  \|\bm{T}_{ \alpha_0}\|_{{\mathcal B}(L^{p_0}(\R^d;\C^m))}e^{\frac{|s|\pi}2}$. Further, since $\bm{T}_{\alpha_0}=\bm{R}^{\alpha_0,n}$ is bounded on $L^{p_0}(\R^d;\C^m)$ from assumption, we infer that for every $s\in\R$, 
$\|\bm{T}_{\alpha_0+is}\bm f\|_{p_0}\le b_{\alpha_0}(s) \|\bm f\|_{p_0}$ for $b_{\alpha_0}(s) \coloneqq c  \|\bm{R}^{\alpha_0,n}\|_{{\mathcal B}(L^{p_0}(\R^d;\C^m))}e^{\frac{|s|\pi}{2}}$. 
    
Analogously, using the fact that $\bm{T}_{\alpha_1}={\bm R}^{\alpha_1,n}$ is bounded on $L^{p_1}(\R^d;\C^m)$ from assumption, it can be proved that for every $s\in\R$, $\|\bm{T}_{\alpha_1+is}\bm f\|_{p_1}\le b_{\alpha_1}(s) \|\bm f\|_{p_1}$ where $b_{\alpha_1}(s) \coloneqq c  \|\bm{R}^{ \alpha_1,n}\|_{{\mathcal B}(L^{p_1}(\R^d;\C^m))}e^{\frac{|s|\pi}2}$. 
It is clear that the functions $b_{\alpha_0}$ and $b_{\alpha_1}$ satisfy the growth condition \eqref{eq:gb}. Hence the thesis follows from Theorem \ref{theo:vv:stein}. 
\end{proof}

\begin{theo}
Fix $p\in (1,2]$. For every $\f\in C^{\infty}_c(\R^d;\C^m)$, the sequence $(\bm{R}^{\alpha,n}\f)_{n\in\N}$ converges almost everywhere in $\R^d$ for every $\alpha\in [0,1/p]$. 
Moreover, if we denote by 
$V^{\alpha}{\bm A}^{-\alpha}\f$ the previous limit, then $V^{\alpha}{\bm A}^{-\alpha}$ extends to a bounded operator on $L^p(\R^d;\C^m)$.
\end{theo}

\begin{proof}
We split the proof into two steps. In the first one, we prove the convergence of $\bm{R}^{\alpha,n}\f$, as $n$ tends to $\infty$, for every $\f\in C^{\infty}_c(\R^d;\C^m)$. Then, in Step 2, we conclude the proof.

{\em Step 1}. 
Fix $f\in C^\infty_c(\R^d;\C)$ and $j\in\{1,\ldots,m\}$, $p\in (1,2]$ and $\alpha\in (0,1/p]$. By
\eqref{lene} we have 
\begin{align*}
((n^{-1} + \bm{A})^{-\alpha}(f{\bm e}_j))(x)
= \frac{1}{\Gamma(\alpha)}\int_0^\infty t^{\alpha-1}e^{-\frac{t}{n}}\bm{w}_A(t,x)dt,\qquad\;\,\textit{a.e.}\ x\in\R^d,
\end{align*}
where $\bm{w}_A=(w_{A,1},\ldots,w_{A,m})$ and
\begin{eqnarray*}
w_{A,i}(t,x)=\int_{\R^d}q_{ij}(t,x,y)f(y)dy,\qquad\;\,i\in\{1,\ldots,m\}.
\end{eqnarray*}
Moreover, by Proposition \ref{proposition:measurability} there exists a measurable set $E$ with complement of zero measure such that, for every $t\in E$, there exists a measurable set $F_t\subset\R^d$ with complement of zero measure satisfying
$(e^{-t\bm{A}}(f\bm{e}_j))(x)=\bm{w}_A(t,x)$ for $t\in E$ and $x\in F_t$.

From \eqref{eq:diamagnetic:inequality} we deduce that 
\begin{align}
t^{\alpha-1}e^{-\frac{t}{n}}\|(V_n(x))^{\alpha}(e^{-t\bm{A}}(f{\bm e}_j))(x)\|\le &t^{\alpha-1}e^{-\frac{t}{n}}\|(V_n(x))^{\alpha}\|\|(e^{-t\bm{A}}(f{\bm e}_j))(x)\|\notag\\
\le &\kappa^{\alpha}t^{\alpha-1}(\lambda_V(x))^{\alpha}(e^{-tH}|f|)(x)
\label{stima-conv}
\end{align}
for every $t>0$, $n\in\N$, and almost every $x\in\R^d$. In analogy with Proposition \ref{proposition:measurability}, we denote by $w_H:(0,\infty)\times\R^d\to\R$ a measurable function such that there exists a measurable set $E'\subset (0,\infty)$ with complement of zero measure and, for every $t\in E'$, a measurable subset $F_t'$ of $\R^d$, with complement of measure zero satisfying 
$(e^{-tH}|f|)(x)=w_H(t,x)$ for every $t\in E'$, $x\in F_t'$.

Combining the previous properties, we conclude that there exist $E''\subset (0,\infty)$, with complement of measure zero and, for every $t\in E''$, a measurable set $F_t''\subset\R^d$, with complement of zero measure such that
\begin{align*}
t^{\alpha-1}e^{-\frac{t}{n}}\|(V_n(x))^{\alpha}\bm{w}_A(t,x)\|
\le \kappa^{\alpha}t^{\alpha-1}(\lambda_V(x))^{\alpha}w_H(t,x)
\end{align*}
for every $t\in E''$ and every $x\in F_t''$.
Clearly, for $t\in E''$, $t^{\alpha-1}e^{-\frac{t}{n}}V_n^{\alpha}\bm{w}_A(t,\cdot)$ tends to $t^{\alpha-1}V^{\alpha}\bm{w}_A(t,\cdot)$, pointwise in $F_t''$ as $n$ tends to $\infty$. 
We claim that the function on the right-hand side of \eqref{stima-conv} belongs to $L^1((0,\infty))$ for almost every $x\in\R^d$. To prove the claim, we observe that, 
from \cite[Theorem 2.2]{KW25}, it follows that there exists a positive constant $c$, independent of $n$ and $f$, such that
\begin{align}
\|\lambda_V^\alpha\chi_{\{n^{-1}<\lambda_V<n\}}(n^{-1}+H)^{-\alpha}f\|_p\leq c\|f\|_p, \qquad\;\, n\in\N.
\label{stima-KW25}
\end{align}
Moreover,  $t^{\alpha-1}e^{-\frac{t}{n}}w_H(t,x)$ monotonically converges to $t^{\alpha-1}w_H(t,x)$ as $n$ tends to infinity for every $t\in E''$ and every $x\in F_t''$. 
Hence, for almost every $x\in\R^d$,
\begin{eqnarray*}
((n^{-1}+H)^{-\alpha})(x)=\frac{1}{\Gamma(\alpha)}
\int_0^{\infty}t^{\alpha-1}e^{-\frac{t}{n}}w_H(t,x)dt
\end{eqnarray*}
tends monotonically to $\int_0^{\infty}t^{\alpha-1}w_H(t,x)dt$. 

From \eqref{stima-KW25} and Fatou's lemma we infer that
\begin{align*}
\bigg\|\lambda_V^\alpha\int_0^\infty t^{\alpha-1}w_H(t,\cdot)dt\bigg\|_p\leq c\|f\|_p.    
\end{align*}
Hence, for almost every $x\in\R^d$ 
the function $t\mapsto t^{\alpha-1}(\lambda_V(x))^{\alpha}w_H(t,x)$ belongs to $L^1((0,\infty))$.
We can thus apply the dominated convergence theorem and conclude that
\begin{align*}
({\bm R}^{\alpha,n}(f{\bm e}_j))(x)=&(V_n(x))^{\alpha}((n^{-1}+\bm{A})^{-\alpha}(f\bm e_j))(x)\\
= & \frac{1}{\Gamma(\alpha)}\int_0^\infty t^{\alpha-1}e^{-\frac{t}{n}}(V_n(x))^{\alpha}\bm{w}_A(t,x)dt
\end{align*}
converges for almost every in $x\in\R^d$ as $n$ tends to $\infty$, and the limit is denoted by $V^\alpha {\bm A}^{-\alpha}(f{\bm e}_j)$. 
Finally, for a general $\bm{f}\in C^{\infty}_c(\R^d;\C^m)$, the convergence follows by linearity, writing $\bm{f}=\sum_{j=1}^mf_j\bm{e}_j$.

{\em Step 2}. Fix $p\in(1,2]$ and $n\in\N$. Proposition \ref{prop:bdd:R:a,e:1/p} yields $\bm{R}^{\frac1p,n}\in {\mathcal B}(L^p(\R^d;\C^m))$, while $\bm{R}^{0,n}=\pmb{\operatorname{Id}}$ belongs to ${\mathcal B}(L^{p}(\R^d;\C^m))$ by construction. Moreover, for both $\bm{R}^{\frac1p,n}$ and $\pmb{\operatorname{Id}}$ the operator norm in ${\mathcal B}(L^p(\R^d;\C^m))$ is bounded by a constant which is independent of $n$. Applying Proposition \ref{prop:KW:gen} to $\alpha_0=0$ and $\alpha_1=1/p$, we infer that $\bm{R}^{\alpha,n}$ is a bounded operator in $L^p(\R^d;\C^m)$ for every $\alpha\in [0,1/p]$ and its operator norm satisfies $\|\bm{R}^{\alpha,n}\|_{{\mathcal B}(L^{p}(\R^d;\C^m))}\le c(\alpha,p)$, which is a positive constant independent of $n$. 

Fix now $\bm f\in C^\infty_c(\R^d;\C^m)$ and $\alpha\in [0,1/p]$. Since $\bm{R}^{\alpha,n}\f$ converges to $V^{\alpha}\bm{A}^{-\alpha}f$ pointwise almost everywhere in $\R^d$, as $n$ tends to $\infty$, by Fatou's lemma we have
\begin{align}
\label{V-riesz}
\int_{\R^d}\|V^{\alpha}\bm{A}^{-\alpha}\bm f\|^pdx
&\le  \liminf_{n\to\infty}\int_{\R^d}\|\bm{R}^{\alpha,n}\f\|^pdx\le c(\alpha,p)\|\bm f\|_p^p.
\end{align}
By density, we conclude that \eqref{V-riesz} holds true for every $\bm f\in L^p(\R^d;\C^m)$.
\end{proof}

\begin{rem}
If $\alpha=\frac12$, then for every $\bm f\in L^2(\R^d;\C^m)$ the sequence $(\bm{R}^{\frac{1}{2},n}\bm f)_{n\in\N}$ converges in $L^2(\R^d;\C^m)$ to $V^{\frac12}\bm{A}^{-\frac12}\f$. Indeed, the definition of $\mathfrak a$ gives $\|V_n^\frac12 \bm f\|_2^2
\leq  \|(n^{-1}+{\bm A})^\frac12 \bm f\|_2^2$ for every $\bm f\in D(\mathfrak a)$, which implies
$\|\bm{R}^{\frac{1}{2},n}\bm f\|_2^2
\leq \|\bm f\|_2^2$ for every $\bm f\in L^2(\R^d;\C^m)$ and every $n\in\N$. Since $\|\bm{R}^{\frac{1}{2},n}\bm f\|_2^2\leq \|{\bm A}^{\frac{1}{2}}(n^{-1}+\bm{A})^{-\frac12}\bm f\|_2^2$ and we have already proved that $(\bm{A}^\frac12(n^{-1}+\bm{A})^{-\frac12}\bm f)_{n\in\N}$ converges in $L^2(\R^d;\C^m)$, we infer that 
$(\bm{R}^{\frac{1}{2},n}\bm f)_{n\in\N}$ is a Cauchy sequence in $L^2(\R^d;\C^m)$, so that it converges.
In particular, is this case the limit operator $V^{\frac12}\bm{A}^{-\frac12}$ is a contraction on $L^2(\R^d;\C^m)$.
\end{rem}

\appendix

\section{Basic results}
\label{appendix}

In this appendix, we collect some basic results that have been used in the paper. 
For the sake of completeness, we provide detailed proofs.

\subsection{A measure theoretical result}
Here, we state and prove the following result which is essentially shown in \cite[Theorem 4.2]{Neidhardt}, see also \cite[Lemma 2.2]{RS99}. We recall that, for $E\subset\R$ and $p\in[1,\infty)$, $C_b(E;L^p(\R^d;\C))$ denotes the set of all functions $u\colon E\to L^p(\R^d;\C)$ that are bounded and continuous.

\begin{prop}
\label{lemma:rappr_misurabile}
Fix $p\in[1,\infty)$. For every function $u\in C_b([0,\infty);L^p(\R^d;\C))$ there exists a function $\widetilde u:(0,\infty)\times \R^d\to\C$, belonging to $\widetilde u\in L^p((0,T)\times \R^d;\C)$ for every $T>0$, such that $u(t)=\widetilde u(t,\cdot)$ almost everywhere in $\R^d$, for almost every $t\in[0,\infty)$. Further, for every measurable function $v:(0,\infty)\times \R^d\to \C$ which satisfies $\widetilde u=v$ almost everywhere in $(0,\infty)\times \R^d$ and every $f\in L^1((0,\infty);\C)$, it holds that
\begin{align}
\label{ug_int_punt}
\int_0^\infty f(t)v(t,x)dt=\bigg (\int_0^\infty f(t)u(t)dt\bigg )(x), \qquad\;\, \textit{a.e. }x\in\R^d.    
\end{align}
\end{prop}

\begin{proof}
Fix $p\in[1,\infty)$, $u\in C_b([0,\infty);L^p(\R^d;\C))$, $T>0$. Since $u\in L^p((0,T);L^p(\R^d;\C))$ there exists a sequence of simple functions
\begin{align*}
u_n(t)=\sum_{i=1}^{m_n}\mathds 1_{B_i}(t)\psi_i, \qquad\;\, t\in(0,T), \;\, B_i\in\mathscr{L} ((0,T)),\;\, \psi_i\in L^p(\R^d;\C), \;\, n\in\N,      
\end{align*}
such that $u_n$ converges to $u$ in $L^p((0,T);L^p(\R^d))$. For every $n\in\N$, we set
\begin{align*}
\widetilde u_n(t,x)=u_n(t)(x)=\sum_{i=1}^{m_n}\mathds 1_{B_i}(t)\psi_i(x), \qquad\;\, (t,x)\in (0,T)\times \R^d,    
\end{align*}
and prove that $\widetilde u_n$ converges in $L^p((0,T)\times \R^d;\C)$ for every $T>0$. To this end, we fix $T>0$ and observe that, since 
\begin{align*}
\int_{(0,T)\times \R^d}|\widetilde u_n(t,x)-\widetilde u_m(t,x)|^pdtdx
= \int_0^T\|u_n(t)- u_m(t)\|_p^pdt
\end{align*}
for every $m,n\in\N$ and $u_n$ converges to $u$ in $L^p((0,T);L^p(\R^d;\C))$, it follows that $(\widetilde u_n)$ is a Cauchy sequence in $L^p((0,T)\times \R^d;\C)$ and therefore it converges to a function $\widetilde u\in L^p((0,T)\times \R^d;\C)$. Further,
\begin{align*}
\lim_{n\to\infty}\int_0^T\|u_n(t)-\widetilde u(t,\cdot)\|_{L^p(\R^d;\C)}^pdt
= \lim_{n\to\infty}\int_{0}^Tdt\int_{\R^d}|\widetilde u_n(t,x)-\widetilde u(t,x)|^pdx=0,  
\end{align*}
so that, for almost every $t\in(0,T)$, $\widetilde u(t,\cdot)=u(t)$ almost everywhere in $\R^d$. The arbitrariness of $T>0$ yields the first part of the statement. 

To complete the proof, we start by noticing that it is sufficient to prove the statement for $\widetilde u$ since
for almost every $x\in\R^d$, the functions $\widetilde u(\cdot,x)$ and $v(\cdot,x)$ coincide almost everywhere in $(0,\infty)$. 

Let $p\in[1,\infty)$ and $B\in\mathscr{L}((0,\infty))$ be a bounded set. Notice that, from definition, for every $n\in\N$ and almost every $x\in\R^d$,
\begin{align*}
\bigg (\int_0^\infty\mathds 1_B(t) u_n(t)dt\bigg )(x)dx
= &  \int_0^\infty\mathds 1_B(t)u_n(t)(x)dt 
= \int_{0}^\infty\mathds 1_B(t)\widetilde u_n(t,x)dt.
\end{align*}
If $T>0$ is such that $B\subset[0,T)$, then
\begin{align*}
\lim_{n\to\infty}\left\|\int_0^\infty\mathds 1_B(t)(u_n(t)-u(t))dt \right\|_p
\leq & \lim_{n\to\infty}\int_0^T\|u_n(t)-u(t)\|_pdt=0.
\end{align*}
By applying Minkowski's integral inequality, we deduce that
\begin{align*}
& \lim_{n\to\infty}\left\|\int_{0}^\infty\mathds 1_B(t)\widetilde u_n(t,\cdot)dt-\int_{0}^\infty\mathds 1_B(t)\widetilde u(t,\cdot)dt\right\|_p^p  \\
= &\lim_{n\to\infty} \int_{\R^d}\left|\int_0^\infty \mathds 1_B(t)(\widetilde u_n(t,x)-\widetilde u(t,x))dt\right|^pdx \\
\leq & |B|^{\frac{p}{p'}} \lim_{n\to\infty}\int_{(0,T)\times\R^d}|\widetilde u_n(t,x)-\widetilde u(t,x)|^pdtdx=0,
\end{align*}
with the convention $|B|^{\frac{1}{\infty}}:=1$.
It follows that 
\begin{align*}
\bigg (\int_0^\infty\mathds 1_B(t)u(t)dt\bigg )(x)=\int_0^\infty\mathds 1_B(t)\widetilde u(t,x)dt,\qquad\;\,\textit{a.e. }x\in\R^d. 
\end{align*}
We have proved that \eqref{ug_int_punt} holds true for every characteristic function $f$ with bounded support and, by linearity for every
simple function $f$ with bounded support.

Let us now fix $f\in L^1((0,\infty);\C)$ and choose a sequence $(f_n)$ of simple functions with bounded support which converges to $f$ in $L^1((0,\infty))$. Then, 
\begin{align*}
\lim_{n\to\infty}\left\|\int_0^\infty f_n(t)u(t)dt-\int_0^\infty f(t)u(t)dt\right\|_p
\leq & \lim_{n\to\infty}\int_0^\infty |f_n(t)-f(t)|\|u(t)\|_pdt \\
\leq & \sup_{t\in[0,\infty)}\|u(t)\|_p\lim_{n\to\infty}
\|f_n-f\|_1
=0.
\end{align*}
Similarly, by applying Minkowski's integral  inequality,
\begin{align*}
\lim_{n\to\infty}\left\|\int_0^\infty f_n(t)\widetilde u(t,\cdot)dt-\int_0^\infty f(t)\widetilde u(t,\cdot)dt \right\|_p 
\leq &\lim_{n\to\infty}\int_0^\infty\|\widetilde u(t,\cdot)\|_p|f_n(t)-f(t)|dt \\
\leq & \esssup_{t\in(0,\infty)}\|\widetilde u(t,\cdot)\|_p\lim_{n\to\infty}\|f_n-f\|_1\\
= & \sup_{t\in(0,\infty)}\|u(t)\|_p\lim_{n\to\infty}\|f_n-f\|_1=0.
\end{align*}
Since
\begin{align*}
\bigg (\int_0^\infty f_n(t)u(t)dt\bigg )(x)
= \int_0^\infty f_n(t)\widetilde u(t,x)dt, \qquad\;\, \textit{a.e. }x\in\R^d,    
\end{align*}
for every $n\in\N$, the assertion follows.
\end{proof}

\subsection{Some properties of \texorpdfstring{$B_q$}{Bq}-functions}

In this subsection, we show that, if $w\in B_q$ for some $q\in(1,\infty)$, then $w\wedge n\in B_q$ and $[w\wedge n]_{B_q}$ can be estimated uniformly with respect to $n\in\N$. For this purpose, we begin by recalling some preliminary results about Muckenhoupt weights. In particular, a nonnegative locally integrable function $w$ belongs to $A_r$, the Muckenhoupt class with exponent $r\in(1,\infty)$, if 
\begin{equation*}
    [w]_{A_r}\coloneqq\sup_{Q}\bigg (\frac 1{|Q|}\int_Qw(x)\,dx\bigg )\bigg (\frac1{|Q|}\int_Q w^{-\frac{r'}{r}}(x)\,dx\bigg )^{\frac{r}{r'}}<\infty,
\end{equation*}
where $r'$ is the conjugate exponent of $r$, and $w$ belongs to $A_1$ if
\begin{equation*}
    [w]_{A_1}\coloneqq\sup_{Q}\bigg (\frac 1{|Q|}\int_Qw(x)\,dx\bigg )\|w^{-1}\|_{L^\infty(Q)}<\infty.
\end{equation*}
Here and in the following with $Q$ we denotes a cube of $\R^d$ with sides parallel to the axes.

For later use, for $a\in (0,1)$, $\gamma, r>0$, we introduce the functions
$\mathfrak{f}_{a,\gamma,r}, \mathfrak{g}_{a,r}:[1,\infty)\to\mathbb\R$, defined by
\begin{equation*}
    \mathfrak{f}_{a,\gamma,r}(t)\coloneqq1+\frac{(2^da^{-1})^\gamma}{1-(2^da^{-1})^\gamma(1-(1-a)^rt^{-1})},\qquad\;\,\mathfrak{g}_{a,r}(t)\coloneqq\frac{\log(t)-\log(t-(1-a)^r)}{d\log (2)-\log(a)}
\end{equation*}
for every $t\in [1,\infty)$.  

\begin{lemma}
\label{lem:prop:Ap}
For every $w\in A_r$ $(r\in [1,\infty))$, the following properties hold:
\begin{enumerate}[\rm(i)]
\item 
$w\wedge k\in A_r$ and $[w\wedge k]_{A_r}\leq (1\vee 2^{r-2})([w]_{A_r}+1)$
for every $k>0$;
\item 
for every $a\in (0,1)$ and every $\gamma\in (0,\mathfrak{g}_{a,r}([w]_{A_r}))$, it holds that
\begin{align}
\label{reverse_holder_gamma}
\left(\frac{1}{|Q|}\int_Q(w(x))^{1+\gamma}dx\right)^{\frac1{1+\gamma}}\leq \frac{\mathfrak{f}_{a,\gamma,r}([w]_{A_r})}{|Q|}\int_Qw(x)dx   \end{align}
for every cube $Q$;
\item 
for every cube $Q$ and every measurable subset $A\subseteq Q$, it holds that
\begin{align}
\label{ineq_sets}
\int_Aw(x)dx\leq \mathfrak{f}_{a,\gamma,r}([w]_{A_r})\left(\frac{|A|}{|Q|}\right)^{\frac{\gamma}{1+\gamma}}\int_Qw(x)dx
\end{align}
for every $\gamma$ as in $(ii)$.
\end{enumerate}    
\end{lemma}

\begin{proof}
(i) Let $w\in A_r$, $k>0$ and set $w_k:=w\wedge k$.

We first consider the case $r\in (1,\infty)$. Fix a cube $Q$. It is immediate to check,
\begin{align*}
\frac{1}{|Q|}\int_Q w_k^{-\frac{r'}{r}}dx
\le \frac{1}{|Q|}\int_Q w^{-\frac{r'}{r}}dx
+ k^{-\frac{r'}{r}}.
\end{align*}

Using the inequality
$(a+b)^{r-1} \le (1\vee 2^{r-2})(a^{r-1}+b^{r-1})$, which holds for every $a,b>0$,
we deduce that
\begin{align*}
\bigg (\frac{1}{|Q|}\int_Q w_k^{-\frac{r'}{r}}dx\bigg )^{r-1}
\le (1\vee 2^{r-2})\bigg [\bigg (\frac{1}{|Q|}\int_Q w^{-\frac{r'}{r}}dx\bigg )^{r-1}+ \frac{1}{k}\bigg ].
\end{align*}
Observing that
\begin{align*}
\frac{1}{|Q|}\int_{Q}w_kdx\le k\wedge \frac{1}{|Q|}\int_Qwdx,
\end{align*}
we deduce that
\begin{align*}
&\frac{1}{|Q|}\int_{Q}w_kdx \bigg (\frac{1}{|Q|}\int_Q w_k^{-\frac{r'}{r}}dx\bigg )^{r-1}\\
\le &(1\vee 2^{r-2})\left(k\wedge \frac{1}{|Q|}\int_{Q}wdx\right)\bigg [\bigg (\frac{1}{|Q|}\int_Q w^{-\frac{r'}{r}}dx\bigg )^{r-1}+ \frac{1}{k}\bigg ]\\
\le &(1\vee 2^{r-2})([w]_{A_r}+1).
\end{align*}
Taking the supremum over all the cubes $Q$ yields the claim.

The case $r=1$ is easier. Fix a cube $Q$. If $\|w^{-1}\|_{L^{\infty}(Q)}\le k$, then 
\begin{align*}
\bigg (\frac{1}{|Q|}\int_Qw_kdx\bigg )\|w_k^{-1}\|_{L^{\infty}(Q)}=&\bigg (\frac{1}{|Q|}\int_Qw_kdx\bigg )\|w^{-1}\|_{L^{\infty}(Q)}\\
\le &\bigg (\frac{1}{|Q|}\int_Qwdx\bigg )\|w^{-1}\|_{L^{\infty}(Q)}\leq [w]_{A_1}.
\end{align*}

If, instead, $\|w^{-1}\|_{L^\infty(Q)}>k$, then $w_k=k$ almost everywhere on $Q$, and hence
\begin{eqnarray*}
\bigg (\frac{1}{|Q|}\int_Qw_kdx\bigg )\|w^{-1}_k\|_{L^\infty(Q)}=1.
\end{eqnarray*}

Summarizing, we have proved that
\begin{align*}
\bigg (\frac{1}{|Q|}\int_Qw_kdx\bigg )\|w_k^{-1}\|_{L^\infty(Q)}\le [w]_{A_1}+1.
\end{align*}
Taking the supremum over all cubes yields the claim also in this case.

(ii)-(iii). These properties follow from \cite{grafakos-modern} (see Theorems 9.2.2 \& 9.3.3).
\end{proof}

Based on Lemma \ref{lem:prop:Ap}, we can now prove that $w\wedge n$ belongs to $B_q$ if $w$ does.

\begin{lemma}
\label{lemma:Bp_truncate}
The following properties are satisfied.
\begin{enumerate}[\rm (i)]
\item
If $w\in B_q$ for some $q\in(1,\infty)$, then $w\wedge n\in B_q$ for every $n\in\N$ and $[w\wedge n]_{B_q}$ is uniformly bounded with respect to $n\in\N$. 
\item 
If $w\in B_\infty$ then $w\wedge n\in B_q$ for every $q\in(1,\infty)$ and $[w\wedge n]_{B_q}$ can be estimated independently of $n\in\N$.
\item 
For every $n\in\N$ and every $\varepsilon>0$, the function $(w\wedge n)+\varepsilon$ belongs to $B_q$ and $[(w\wedge n)+\varepsilon]_{B_q}$ is bounded uniformly with respect to $n\in\N$ and $\varepsilon>0$. 
\end{enumerate} 
\end{lemma}

\begin{proof}
Fix $q\in (1,\infty)$ and $w\in B_q$. By \cite{gehring}, it follows that there exists $s>q$ such that $w\in B_s$. Hence, $w^q\in B_{s/q}$, since by applying H\"older's inequality it follows that
\begin{align*}
\left(\frac1{|Q|}\int_Q ((w(x))^q)^{\frac{s}{q}}dx\right)^{\frac{q}{s}}
\leq \left(\frac{[w]_{B_s}}{|Q|}\int_Q w(x)dx\right)^q
\leq \frac{[w]_{B_s}^q}{|Q|}\int_Q(w(x))^qdx
\end{align*}
for every cube $Q$.  
This implies that $w^q\in A_r$ for some $r\in[1,\infty)$ (see \cite[Theorem 9.3.3]{grafakos-modern}). From Lemma \ref{lem:prop:Ap}(i), we infer that $(w\wedge n)^q$ belongs to $A_r$ as well and $[(w\wedge n)^q]_{A_r}\le (1\vee 2^{r-2})([w^q]_{A^r}+1)$.

Next, we claim that there exist positive constants $\overline\gamma$ and $\overline{c}$, independent of $n$, such that \eqref{reverse_holder_gamma} and \eqref{ineq_sets} are satisfied with $\gamma=\overline\gamma$ and with $\mathfrak{f}_{a,\gamma,r}([w]_{A_r})$ and $[w]_{A_r}$ being replaced, respectively, by $\overline{c}$ and $[w\wedge n]_{A_r}$ for every $n\in\N$. Indeed, for every $a\in(0,1)$, the function $\mathfrak{g}_{a,r}$ is decreasing, while the function $\mathfrak{f}_{a,\gamma,r}$ is increasing. Hence, if we fix $\overline a\in(0,1)$ and choose 
$\overline\gamma=\mathfrak{g}_{\overline a,r}((1\vee 2^{r-2})([w^q]_{A_r}+1)+1)$ and
$\overline c=\mathfrak{f}_{\overline{a},\overline\gamma,r}((1\vee 2^{r-2})([w^q]_{A_r}+1)+1)$, then from Lemma \ref{lem:prop:Ap} 
we get the claim.

We can now prove the three properties in the statement of the lemma.

(i) We fix $n\in\N$, $K>0$ such that $K^{\frac{\overline\gamma}{1+\overline \gamma}}\geq 2\overline c$  and, for every cube $Q$, we introduce the set $E_Q\coloneqq\{x\in Q:(w(x)\wedge n)>K{\rm av}_Q(w\wedge n)\}$. From Chebyshev's inequality, it follows that
\begin{align}
\label{stimaE_Q}
|E_Q|\leq \frac{1}{K{\rm av}_Q(w\wedge n)}\int_Q(w(x)\wedge n)dx= K^{-1}|Q|.  \end{align}
Hence, \eqref{ineq_sets}, with $(w\wedge n)^q$ instead of $w$, applied to the set $E_Q$ and \eqref{stimaE_Q}, yield
\begin{align*}
\int_{E_Q}(w(x)\wedge n)^q dx
\le  \overline cK^{-\frac{\overline \gamma}{1+\overline \gamma}}\int_{Q}(w(x)\wedge n)^qdx\le\frac{1}{2}\int_{Q}(w(x)\wedge n)^qdx. 
\end{align*}
It thus follows that
\begin{align*}
\int_{Q}(w(x)\wedge n)^qdx
= & \int_{E_Q}(w(x)\wedge n)^qdx+\int_{Q\setminus E_Q}(w(x)\wedge n)^qdx \\
\leq & \frac12 \int_{Q}(w(x)\wedge n)^qdx
+(K{\rm av}_Q(w\wedge n))^q|Q\setminus E_Q|,
\end{align*}
which gives
\begin{align*}
\int_Q(w(x)\wedge n)^qdx\leq 2|Q|\left(\frac{K}{|Q|}\int_Q (w(x)\wedge n)dx\right)^q.    
\end{align*}
We conclude that
\begin{align*}
\left(\frac{1}{|Q|}\int_Q(w(x)\wedge n)^qdx\right)^{\frac1q}\leq \frac{2^{\frac1q}K}{|Q|}\int_{Q}(w(x)\wedge n) dx,
\end{align*}
which means that $w\wedge n\in B_q$, with $[w\wedge n]_{B_q}\leq 2^{\frac1q}K$ for every $n\in\N$.

(ii) If $w\in B_\infty$, then $[w]_{B_q}\leq [w]_{B_\infty}$ for every $q\in(1,\infty)$ and the statement follows from the first part of the lemma.

(iii) It follows immediately from Minkowski's inequality
that $(w\wedge n)+\varepsilon$ belongs to $B_q$ and
$[(w\wedge n)+\varepsilon]_{B_q}\le [w\wedge n]_{B_q}$ for every $\varepsilon>0$. 
Since $[w\wedge n]_{B_q}$ can be estimated uniformly with respect to $n\in\N$, the assertion follows at once.
\end{proof}

\subsection{Some properties of Schr\"odinger operators with bounded potential}
In the following lemma, we recall some classical results related to the operator $-\Delta+v$, when $v$ is a measurable and bounded function. These results are used to prove the maximal inequalities for the magnetic Schr\"odinger operator in Section \ref{sec:max_ineq_Lp}.

We consider the realization $T_p$ $(p\in [1,\infty))$ of the formal operator $-\Delta+v$, with domain $D(T_p)=W^{2,p}(\R^d;\C)$, if $p\in (1,\infty)$, and $D(T_1)=\{u\in L^1(\R^d;\C): \Delta u\in L^1(\R^d;\C)$\}.

\begin{lemma}
\label{lemma:operatori_scalari_n}
The following facts are true.    
\begin{enumerate}[\rm(i)]
\item 
For every $p\in[1,\infty)$, the operator $-T_p$
generates a positive strongly continuous analytic semigroup $(e^{-tT_p})_{t\ge 0}$ in $L^p(\R^d;\C)$. The semigroups $(e^{-tT_p})_{t\ge 0}$ are consistent with each other.
\item 
For every $u\in D(T_1)$ it holds
\begin{align}
\label{max_ineq_1_n}
\|vu\|_1\leq \|T_1u\|_1, \qquad \|\Delta u\|_1\leq 2\|T_1u\|_1,
\end{align}
and, if $v\in B_q$ for some $q\in (1,\infty)\cup\{\infty\}$, then there exist a positive constant $c_p$ which depends on $[v]_{B_q}$ such that
\begin{equation}
\|vu\|_p+\|\Delta u\|_p\leq c_p\|T_pu\|_p , \qquad u\in D(T_p),
\label{max_ineq_p_n}
\end{equation}
for every $p\in [1,q]$, if $q<\infty$, and for every $p\in [1,\infty)$, if $q=\infty$. 
\end{enumerate}
\end{lemma}
\begin{proof}
(i) It follows from classical results of the semigroup theory.

(ii) Estimate \eqref{max_ineq_1_n} follows from \cite{kato:1986} and the fact that $v$ is bounded, whereas
estimate \eqref{max_ineq_p_n} follows from \cite[Theorem 1.1]{auscher-benali:2007}. Notice that the constant $c_p$ only depends on $p$, $d$ and $[v]_{B_q}$.
\end{proof}

\subsection{An interpolation result}
Next, we prove an interpolation result for analytic families of vector-valued operators,
which a crucial tool in Section \ref{sec:0:RT}.
In the scalar case, this goes back to the work of Stein \cite{stein56} and it has been extended first to analytic families of operators, taking values in Banach spaces, by Cwikel and Sagher \cite{CwSa88}, and then, again, to analytic families of multilinear operators by Grafakos and  Masty\l{}o \cite{GrMa14}. 
Here, we present a version adapted to our context, which is not a straightforward consequence of the above results, and, for the sake of completeness, provide a detailed proof.
	
First of all, we need to introduce the notion of analytic function of admissible growth.
From now on, for every $\alpha_0, \alpha_1\in\R$, with $\alpha_0<\alpha_1$, we set $\Pi_{\alpha_0}^{\alpha_1}=\{z\in\C\colon \alpha_0<\operatorname{Re}(z)<\alpha_1\}$  and denote by $\overline{\Pi}_{\alpha_0}^{\alpha_1}$ its closure. When, $\alpha_0=0$ and $\alpha_1=1$, we simply write $\Pi$ and $\overline\Pi$.

\begin{defi} 
Let $\alpha_0,\alpha_1\in\R$ be such that $\alpha_0<\alpha_1$. A function $F:\overline\Pi_{\alpha_0}^{\alpha_1}\to\C$, which is analytic in $\Pi_{\alpha_0}^{\alpha_1}$ and continuous in $\overline\Pi_{\alpha_0}^{\alpha_1}$, is called 
of admissible growth if  there exist constants $a\in (0,\pi)$ and $c>0$ such that 
\begin{equation*}
\log(|F(z)|)\le ce^{\frac{a|\operatorname{Im}(z)|}{\alpha_1-\alpha_0}},\qquad\;\,z\in\overline\Pi_{\alpha_0}^{\alpha_1}.
\end{equation*}
\end{defi}

\begin{defi} 
Let $\alpha_0,\alpha_1\in\R$ be such that $\alpha_0<\alpha_1$. A family 
$\{\bm{T}_z: z\in\overline\Pi_{\alpha_0}^{\alpha_1}\}$ of linear operators which map simple $\C^m$-valued functions defined in $\R^d$ with bounded support into measurable $\C^m$-valued functions, still defined in $\R^d$, is  called of admissible growth if for every pair of simple functions  $\f$ and $\g$, the map $\Phi_{\f,\g}\colon\overline \Pi_{\alpha_0}^{\alpha_1}\to\C$, defined by
\begin{equation*}
\Phi_{\f,\g}(z)\coloneqq \int_{\R^d}\langle \bm{T}_z\f,\g\rangle_{\C^m} dx, \qquad\;\, z\in \overline\Pi_{\alpha_0}^{\alpha_1},
\end{equation*}
is of admissible growth.
\end{defi}

To prove the interpolation theorem, we need the following lemma, which is an improvement of the well-known Hadamard's three lines lemma, due to I. Hirschman \cite[Lemma 1.3.8]{grafakos-classic}.
	
\begin{lemma}\label{lem:Hadamard}
Let $F:\overline\Pi\to\C$ be a function of admissible growth. Assume that there exist positive real measurable functions $M_0, M_1$ such that 
$|F(iy)|\le M_0(y)$ and $|F(1+iy)|\le M_1(y)$
for every $y\in\R$. Then, for $x\in (0,1)$, it holds that
\begin{equation}\label{eq:stima:F}
|F(x)|\le\exp{\left\{\frac{\sin(\pi x)}{2}\int_{\R}\left[\frac{\log(M_0(y))}{\cosh(\pi y)-\cos(\pi x)}+\frac{\log(M_1(y))}{\cosh(\pi y)+\cos(\pi x)}\right]dy\right\}}.
\end{equation}
\end{lemma}
	
We can now prove the interpolation result.

\begin{theo}\label{theo:vv:stein}
Let $\{\bm{T}_z: z\in\overline\Pi_{\alpha_0}^{\alpha_1}\}$ be a family of admissible growth. Fix $p_0,p_1,q_0,q_1\in [1,\infty)\cup\{\infty\}$, $\theta\in (0,1)$ and set
\begin{equation}\label{eq:inter:exponent}
\frac1p = \frac{1-\theta}{p_0} + \frac{\theta}{p_1},\qquad\;\,\frac1q = \frac{1-\theta}{q_0} + \frac{\theta}{q_1}.
\end{equation}
Suppose further that
\begin{equation*}
(i)~\|\bm{T}_{\alpha_0+iy}\f\|_{q_0}\le b_0(y)\|\f\|_{p_0},
\qquad\;\, (ii)~\|\bm{T}_{\alpha_1+iy}\f\|_{q_1}\le b_1(y)\|\f\|_{p_1}
\end{equation*}
for every simple function $\f$ with bounded support, every $y\in\R$ and some measurable functions $b_0,b_1:\R\to (0,\infty)$ such that
\begin{equation}\label{eq:gb}
\log(b_i(y))\le ce^{\frac{b|y|}{\alpha_1-\alpha_0}},\qquad\;\,y\in\R,\;\,i\in\{0,1\},
\end{equation}
for some constants $b\in (0,\pi)$ and $c\in (0,\infty)$.
Then, there exists a positive constant $c'$, depending only on $m$, $\theta$, $p_0$, $q_0$, $p_1$ and $q_1$, such that 
\begin{equation*}
\|\bm{T}_{\alpha_0+(\alpha_1-\alpha_0)\theta}\f\|_{q}\le c'\|\f\|_{p},\qquad\;\,
\bm{f}\in L^p(\R^d;\C^m).
\end{equation*}
\end{theo}

\begin{proof}
Let $p_0,p_1,q_0,q_1$ be as in the statement. We split the proof into two steps. In Step 1, we show that we can reduce ourselves to the case $\alpha_0=0$ and $\alpha_1=1$. In Step 2, we address this case.

{\em Step 1}.
Define 
$\widetilde{\bm T}_z= \bm{T}_{\alpha_0+(\alpha_1-\alpha_0)z}$ for every $z\in\overline\Pi$ Then, 
$\{\widetilde{\bm T}_z: z\in\overline\Pi\}$ is a family of operators of admissible growth. Clearly, the map $\omega\mapsto\widetilde{\bm T}_{\omega}$ is continuous in $\overline{\Pi}$ and analytic in $\Pi$. Moreover,
\begin{align*}
\widetilde{\bm\Phi}_{\bm f,\bm g}(\omega)=\int_{\R^d}\langle \widetilde {\bm T}_\omega\bm f,\bm g\rangle_{\C^m}dx
=\int_{\R^d}\langle {\bm T}_{\alpha_0+\omega(\alpha_1-\alpha_0)}\bm f,\bm g\rangle_{\C^m}dx
=\bm{\Phi}_{\bm f,\bm g}(\alpha_0+\omega(\alpha_1-\alpha_0))   
\end{align*}
for every $\omega\in\Pi$. From this formula and recalling that  
$\bm{\Phi}_{\f,\bm g}$ is of admissible growth, it follows immediately that $\widetilde{\bm\Phi}_{\bm f,\bm g}$ is of admissible growth as well.
Moreover,
\begin{equation*}
\|\widetilde {\bm T}_{iy}\f\|_{q_0}\le b_0((\alpha_1-\alpha_0)y) \|\f\|_{p_0},\qquad\;\, \|\widetilde {\bm T}_{1+iy}\f\|_{q_1}\le b_1((\alpha_1-\alpha_0)y)\|\bm f\|_{p_1}
\end{equation*}
for every $y\in\R$ and every simple function $\f$ with bounded support. 

Next, observe that \eqref{eq:gb} implies
\begin{eqnarray*}
\log(b_j((\alpha_1-\alpha_0)y))\le Ce^{b|y|},\qquad\;\,y\in\R,\;\,j=0,1.   
\end{eqnarray*}
It thus follows that the family $\{\widetilde {\bm T}_z: z\in\overline\Pi\}$ satisfies the assumptions of the theorem with $\alpha_0=0$ and $\alpha_1=1$. Once the theorem is proved in this case, we can conclude that the operator ${\bm T}_{\alpha_0+\theta(\alpha_1-\alpha_0)}$ extends to a bounded operator from $L^p(\R^d;\C^m)$ to $L^q(\R^d;\C^m)$ and its operator norm can be bounded by a positive constant depending only on
$m$, $p_0$, $q_0$, $p_1$ and $q_1$. 

{\em Step 2}. We now prove the theorem in the case $\alpha_0=0$ and $\alpha_1=1$. For this purpose, let $\boldsymbol{\varphi}, {\boldsymbol{\psi}}:\R^d\to\C^m$ be two simple functions with bounded support such that $\|\boldsymbol{\varphi}\|_p=\|\boldsymbol{\psi}\|_{q'}=1$. Then, there exist finitely many pairwise disjoint Lebesgue measurable subsets $E_k$ and $F_h$ of $\R^d$ with finite measure, positive numbers $v_{k,1},\dots,v_{k,m}$ and $w_{h,1},\dots,w_{h,m}$ and real numbers $\theta_{k,j},\delta_{h,j}\in\R$ for $k\in\{1,\dots,r\}$, $h\in\{1,\dots,s\}$ and $j\in\{1,\dots,m\}$ such that
\begin{align*}
\boldsymbol{\varphi}(x) & = \sum_{k=1}^r\sum_{j=1}^mv_{k,j}e^{i\theta_{k,j}}\mathds{1}_{E_k}(x){\bm e}_j =: \sum_{j=1}^m\varphi_j(x){\bm e}_j,\\
\boldsymbol{\psi}(x) &= \sum_{h=1}^s\sum_{j=1}^mw_{h,j}e^{i\delta_{h,j}}\mathds{1}_{F_h}(x){\bm e}_j =: \sum_{j=1}^m\psi_j(x){\bm e}_j
\end{align*}
for every $x\in\R^d$,
where $\bm{e}_j$, $j\in\{1,\ldots,m\}$, is the $j$-th element of the canonical basis of $\C^m$. For $z\in\overline \Pi$, we define 
\begin{equation*}
\beta(z)\coloneqq\left( \frac{1-z}{p_0}+\frac{z}{p_1}\right)p,\qquad\;\,\gamma(z) \coloneqq \left( \frac{1-z}{q_0'} + \frac{z}{q_1'} \right)q',\qquad\;\,\textit{ if }p,q'<\infty,
\end{equation*}
$\beta(z)=1$, if $p=\infty$, and $\gamma(z)=1$ if $q'=\infty$. Moreover, we introduce the functions 
\begin{align*}
{\bm F}_z
=\sum_{j=1}^m\sum_{k=1}^rv_{k,j}^{\beta(z)}e^{i\theta_{k,j}}\mathds{1}_{E_k}{\bm e}_j,\qquad\;\,
{\bm G}_z=
\sum_{j=1}^m\sum_{h=1}^sw_{h,j}^{\gamma(z)}e^{i\delta_{h,j}}\mathds{1}_{F_h}{\bm e}_j.
\end{align*}
From \eqref{eq:inter:exponent}, it follows that $\bm{F}_{\theta} = \boldsymbol{\varphi}$ and $\bm{G}_{\theta} = \boldsymbol{\psi}$.

Let $\bm{\Phi}:\overline{\Pi}\to \C$ be the function defined by
\begin{equation*}
\bm{\Phi}(z) \coloneqq \int_{\R^d}\langle \bm{T}_z\bm{F}_z,\bm{G}_z \rangle_{\C^m} dx, \qquad\;\, z\in\overline{\Pi}.
\end{equation*}
Note that
\begin{align*}
\bm{\Phi}(z) 
&= \sum_{j,\ell=1}^m\sum_{k=1}^r\sum_{h=1}^sv_{k,j}^{\beta(z)}w_{h,\ell}^{\gamma(z)}e^{i(\theta_{k,j}+\delta_{h,\ell})}\int_{\R^d}\langle \bm{T}_z(\mathds{1}_{E_k}{\bm e}_j),\mathds{1}_{F_h}{\bm e}_{\ell}\rangle_{\C^m} dx
\end{align*}
for every $z\in\overline\Pi$.
The assumptions on the family $\{\bm{T}_z: z\in\overline\Pi\}$ and the definition of the functions $\alpha$ and $\gamma$ imply that $\bm{\Phi}$ is of admissible growth. 

Assume that $p,q'<\infty$. Since $\operatorname{Re}(\beta(iy))=p/p_0$ for every $y\in\R$ and the set $E_1,\ldots,E_r$ are pairwise disjoint, recalling that $\|\bm{\varphi}\|_p=1$ we obtain that
\begin{align*}
\int_{\R^d}\|\bm{F}_{iy}\|^{p_0}dx 
\le &c(p_0,m) \sum_{j=1}^m\int_{\R^d}|(\bm{F}_{iy})_j|^{p_0}dx\\
=&c(p_0,m)\sum_{j=1}^m\sum_{k=1}^r\int_{E_k}|v_{k,j}|^{p}dx\notag\\
=& c(p_0,m)\int_{\R^d}\sum_{j=1}^m|\varphi_j|^pdx\le 
c(p_0,m)
\end{align*}
for every $y\in\R$. Arguing similarly, we can show that
\begin{align*}
\|\bm{F}_{1+iy}\|_{p_1} \leq c(p_1,p,m),\qquad
\|\bm{G}_{j+iy}\|_{q_j'} \leq c(q_j',q',m),\qquad\,j=0,1. 
\end{align*}

If $p=\infty$, which means that $p_0=p_1=\infty$, then $\|\bm{F}_{iy}\|_{p_0}=\|\bm{F}_{1+iy}\|_{p_1}= 1$ for every $y\in\R$. Similarly, if $q'=\infty$, then 
$\|\bm{G}_{iy}\|_{q_0'}=\|\bm{G}_{1+iy}\|_{q_1'}=1$.
Therefore, for every $y\in\R$ we get
\begin{align*}
|\bm\Phi(iy)| \le &\int_{\R^d}|\langle \bm{T}_{iy}\bm{F}_{iy},\bm{G}_{iy} \rangle_{\C^m}|dx\\
\le &\|\bm{T}_{iy}\bm{F}_{iy}\|_{q_0}\|\bm{G}_{iy}\|_{q_0'}\\
\le & b_0(y)\|\bm{F}_{iy}\|_{p_0}\|\bm{G}_{iy}\|_{q_0'} \le  c_0 b_0(y)
\end{align*}
and
\begin{align*}
|\bm{\Phi}(1+iy)| 
\le &\|\bm{T}_{1+iy}\bm{F}_{1+iy}\|_{q_1}\|\bm{G}_{1+iy}\|_{q_1'}\\
\le & b_1(y)\|\bm{F}_{1+iy}\|_{p_1}\|\bm{G}_{1+iy}\|_{q_1'} \le  c_1 b_1(y),
\end{align*}
where $c_0$ and $c_1$ are positive constants which depend at most on $p_0$, $p_1$, $q_0$, $q_1$ and $m$.

This proves that $\bm{\Phi}$ satisfies all the assumptions of Lemma \ref{lem:Hadamard} and we can conclude that
\begin{equation}\label{eq:Phi:final:estimate}
\left|\int_{\R^d}\langle \bm{T}_{\theta}\boldsymbol{\varphi},\boldsymbol{\psi}\rangle_{\C^m} dx\right| = |\bm{\Phi}(\theta)|\le b(\theta),
\end{equation}
where
\begin{equation*}
b(\theta)\coloneqq\exp\bigg\{\frac{\sin(\pi\theta)}{2}\int_{\R}\bigg [\frac{\log(c_0 b_0(y))}{\cosh(\pi y)-\cos(\theta\pi)}+\frac{\log(c_1 b_1(y))}{\cosh(\pi y)+\cos(\pi\theta)}\bigg ]dy\bigg \}.
\end{equation*}

Note that the growth condition on the functions $b_i$ ($i\in\{0,1\}$) implies that the right-hand side of \eqref{eq:stima:F} is finite.   

Taking the supremum in \eqref{eq:Phi:final:estimate} over all the simple functions $\boldsymbol{\psi}$ with bounded support and unit $L^{q'}$-norm, we get, using also the linearity of the operator $\bm{T}_{\theta}$, $\|\bm{T}_{\theta}\boldsymbol{\varphi}\|_q\le b(\theta)=b(\theta)\|\boldsymbol{\varphi}\|_p$. The density of the set of all the simple functions with bounded support in $L^p(\R^d;\C^m)$ implies that the operator $\bm{T}_{\theta}$ extends uniquely to a bounded linear operator from $L^p(\R^d;\C^m)$ into $L^q(\R^d;\C^m)$, for all $p$ and $q$ as in \eqref{eq:inter:exponent}.
\end{proof}

\end{document}